\documentclass[10pt,twocolumn,twoside]{IEEEtran}

\usepackage{amsmath,amssymb,amsfonts, amsthm}
\usepackage{mathrsfs}
\usepackage{graphicx}

\usepackage{algorithm}
\usepackage{algorithmic}
\usepackage{setspace}
\let\Algorithm\algorithm
\renewcommand\algorithm[1][]{\Algorithm[#1]\setstretch{1.2}}
\usepackage{textcomp}
\usepackage{array}

\usepackage{scribe}
\usepackage{threeparttable}
\usepackage[tight,footnotesize]{subfigure}

\usepackage{cite}
\def\href#1#2{#2}
\def\BibTeX{{\rm B\kern-.05em{\sc i\kern-.025em b}\kern-.08em
    T\kern-.1667em\lower.7ex\hbox{E}\kern-.125emX}}

\begin{document}

\title{A Unified Dual Consensus Approach to Distributed Optimization with Globally-Coupled Constraints}
\author{Zixuan Liu, Xuyang Wu, Dandan Wang, and Jie Lu, \IEEEmembership{Member, IEEE}
\thanks{Z. Liu and D. Wang is with the School~of~Information Science and Technology, ShanghaiTech University, Shanghai 201210, China (e-mail: \nobreak{liuzx2022@shanghaitech.edu.cn}; wangdd2@shanghaitech.edu.cn).}
\thanks{X. Wu is with the School of System Design and Intelligent Manufacturing (SDIM), Southern University of Science and Technology, Shenzhen 518055, China (e-mail: wuxy6@sustech.edu.cn).}
\thanks{Jie Lu is with the School of Information Science and Technology, ShanghaiTech University and the Shanghai Engineering Research Center of Energy Efficient and Custom AI IC, Shanghai 201210, China (e-mail: lujie@shanghaitech.edu.cn).}
}

\maketitle

\begin{abstract}
    This article explores distributed convex optimization with globally-coupled constraints, where the objective function is a general nonsmooth convex function, the constraints include nonlinear inequalities and affine equalities, and the feasible region is possibly unbounded. To address such problems, a unified DUal Consensus Algorithm (DUCA) and its proximal variant (Pro-DUCA) are proposed, which are unified frameworks that approximate the method of multipliers applied to the corresponding dual problem in no need of a closed-form dual objective. With varied parameter settings, DUCA and Pro-DUCA not only extend a collection of existing consensus optimization methods to solve the dual problem that they used to be inapplicable to, but also aid in offering new efficient algorithms to the literature. The proposed unified algorithms are shown to achieve $O(1/k)$ convergence rates in terms of optimality and feasibility, providing new or enhanced convergence results for a number of existing methods. Simulations demonstrate that these algorithms outperform several state-of-the-art alternatives in terms of objective and feasibility errors.
\end{abstract}


\section{Introduction} \label{sec:introduction}

Distributed optimization algorithms are designed to solve optimization problems over multi-agent systems. The agents aim to collaboratively minimize a global objective function, which is achieved through local computations and communications. Over recent decades, interest in this field has surged, driven by its extensive applications in areas such as distributed machine learning \cite{Ned20}, distributed model predictive control \cite{NAL18}, and network resource allocation \cite{NeS18}. 

This article addresses a challenging problem in this area, namely, {\it distributed optimization with coupled constraints} \eqref{prob:CC}. In such problems, agents must optimize a global objective while their local decision variables are subject to globally-coupled constraints, encompassing the problem structures found in all the aforementioned applications.

A common strategy for addressing problem~\eqref{prob:CC} involves solving its dual problem, which is essentially in the form of {\it consensus optimization} \cite{SLW15a}. Numerous distributed methods have been developed to tackle it, generally falling into two main categories. The first category comprises dual subgradient methods \cite{NeS18, SiJ16, FMG17, LLS20, LWY21, LBN24} and primal-dual subgradient methods \cite{MaC17, CNS14}. These methods typically utilize either weighted averaging \cite{NeS18, SiJ16, FMG17, MaC17, CNS14} or dynamic averaging \cite{LLS20, LWY21, LBN24} to achieve consensus among dual variables. The second category involves methods that re-dualize the consensus constraint on dual variables, thereby introducing additional variables. These methods proceed by either formulating a saddle-point problem, solvable via primal-dual algorithms \cite{HMS24, YXJ24} and the ADMM-based method \cite{AyH19}, or by constructing a Karush-Kuhn-Tucker (KKT) system, which is then tackled using operator-splitting methods \cite{LWY20, HeZ24} or the generalized proximal point algorithm \cite{GoZ23}. In this latter category, the problem is often regarded as distinct from the standard consensus optimization problem, leading to underutilization of the extensive tools developed for consensus optimization.

Recent research efforts have sought to bridge the gap, i.e., solve the dual problem of \eqref{prob:CC} by leveraging established consensus optimization methods. For example, Augmented Lagrangian Tracking \cite{FaP23} employs Proximal-Tracking \cite{FaP22} to address the dual problem. Similarly, IPLUX \cite{WWL23} utilizes a dual variant of P-EXTRA \cite{SLW15b} to manage coupled affine equality constraints. Additionally, \cite{LLS20} proposes a unitary distributed subgradient method designed for consensus optimization and then applies it to the dual of \eqref{prob:CC}. Despite these advancements, the attempts remain unsystematic, and these methodologies exhibit deficiencies in terms of restrictive problem formulations and limited convergence results.

In this article, we advance this line of research by introducing a novel algorithmic framework called the unified DUal Consensus Algorithm (DUCA). By virtue of approximating the method of multipliers applied to the dual problem of \eqref{prob:CC}, DUCA not only serves as a versatile framework for designing new distributed methods targeting \eqref{prob:CC}, but also systematically generalizes a collection of existing algorithms (or their special cases) \cite{FNN20, FaP23, GoZ23, WWL23}, and enables the application of various consensus optimization methods \cite{AWL18, SLW15b, MaO17, HoC17, FaP22} to the dual of \eqref{prob:CC}, which were originally inapplicable (see Section~\ref{sec:problem_formulation}). Additionally, its proximal variant Pro-DUCA is proposed to tackle \eqref{prob:CC} in a more general form by eliminating the compactness assumption on the feasible region. Moreover, both DUCA and Pro-DUCA offer theoretically guaranteed convergence rates that are competitive compared to the current state of the art.

The contributions of this article are summarized as follows.
\begin{enumerate}
    \item The proposed Pro-DUCA is able to tackle the \emph{general} form of \eqref{prob:CC} (with nonsmooth objective, nonlinear inequality constraints, affine equality constraints, and possibly unbounded feasible region), and converges asymptotically with a {\it constant stepsize}. Among all the aforementioned methods, only \cite{AyH19, GoZ23, WWL23, FaP23} achieve these two merits. However, \cite{WWL23} requires the inequality constraints to be Lipschitz continuous, and \cite{AyH19} further requires them to be smooth, noting that neither of such conditions are imposed by Pro-DUCA as well as DUCA.
    
    \item DUCA and Pro-DUCA reach $O(1/k)$ convergence rates of both objective and feasibility errors, which outperform the asymptotic convergence of Augmented Lagrangian Tracking \cite{FaP23}, match the results of DPDA-D \cite{AyH19} and IPLUX \cite{WWL23}, and is comparable to the $o(1/k)$ convergence rate with respect to a first-order optimality residual achieved by DPMM \cite{GoZ23}. 

    \item DUCA broadens the applicability of various established methods for consensus optimization \cite{AWL18, SLW15b, MaO17, HoC17, FaP22} to the dual problem of \eqref{prob:CC}, as these methods require the knowledge of a closed-form dual function. Consequently, we provide new convergence results for their primal recovery process. This approach is more systematic than the previous efforts in \cite{FaP23, WWL23, LLS20}, since each of them only facilitates one specific consensus optimization method in overcoming the above issue.
    \item DUCA and Pro-DUCA also generalize a variety of other existing algorithms (or their special cases) for solving problem~\eqref{prob:CC} \cite{FNN20, FaP23, GoZ23, WWL23}, providing enhanced convergence results for \cite{FNN20, FaP23} and supplementary results for \cite{GoZ23}.
\end{enumerate}

The outline of the article is as follows. Section~\ref{sec:problem_formulation} formulates the problem and provide motivation to the algorithm design. Section~\ref{sec:development} describes the proposed DUCA and Pro-DUCA, along with their distributed implementations. Section~\ref{sec:related} discusses how these algorithms extend and generalize previous methods. Section~\ref{sec:convergence} provides the convergence analysis, while Section~\ref{sec:numerical} presents comparative simulation results. Finally, Section~\ref{sec:conclusion} concludes the study.

\emph{Notation and Definition:} For a nonempty, convex set $X\subseteq\R^n$ and a point $x\in\R^n$, $\calP_X[x]$ denotes the projection of $x$ onto $X$ and $\mathscr{N}_X(x)\subseteq\R^n$ represents the normal cone of $X$ at $x$, i.e., $\mathscr{N}_X(x):=\{y\in\R^n\ |\ y^T(z-x)\leq 0,\ \forall z\in X\}$. If $X=\R_+^n$, we simply denote $[\cdot]_+:=\calP_X[\cdot]$. Besides, for a nonempty set $X\subseteq\R^n$, $\delta_X(\cdot)$ represents its indicator function and $X^\circ$ denotes its polar cone, which is given by $X^\circ:=\{y\in\R^n\ |\ y^T x\leq0,\ \forall x\in X\}$. 

For a matrix $A\in\R^{n\times n}$, we use $A^\dag$ and $\norm{A}$ to denote its Moore-Penrose pseudo-inverse and spectral norm. For matrices $A$ and $B$, $A\otimes B$ is their Kronecker product. For a symmetric positive semidefinite matrix $A\in\R^{n\times n}$, and a vector $z\in\R^n$, $\norm{z}:=\sqrt{z^Tz}$ and $\norm{z}_A:=\sqrt{z^TAz}$. We use $\lambda_i(\cdot)$ to denoted the $i$th largest eigenvalue of a matrix. Besides, $\diag(D_1,\ldots,D_n)$ represents the block diagonal matrix with square matrices $D_1,\ldots,D_n$ being its diagonal blocks.

For a convex function $f:\R^n\rightarrow \R$ and a point $x\in\R^n$, we use $\partial f(x)\subseteq\R^n$ to denote the subdifferential of $f$ at $x$; 
if $f$ is differentiable, then $\nabla f(x)$ is the gradient of $f$ at $x$.

\section{Problem and Motivation}\label{sec:problem_formulation}

This section formulates a class of distributed optimization problems with coupled constraints and provides the motivation of this work. 

We consider a network of $N$ agents, represented by a connected, undirected graph $\calG=(\calV, \calE)$, where $\calV=\{1,2,\ldots,N\}$ is the vertex set and $\calE\subseteq \{\{i,j\}\ |\ i,j\in\calV, i\neq j\}$ is the set of communication links.  The agents attempt to collaboratively solve the following problem over $\calG$ :
\begin{align*}
    \begin{split}
        \underset{x_i\in X_i, \forall i\in\calV}{\text{minimize}}\quad &\sum_{i\in\calV}f_i(x_i) \\
        \text{subject to}\quad &\sum_{i\in\calV} g_i(x_i) \leq \0_{m}, \\
        &\sum_{i\in\calV}h_i(x_i) = \0_p.
    \end{split} \label{prob:CC} \tag{$\calC\calC$}
\end{align*}
In this network-wide optimization problem, each node $i\in\calV$, with its local decision variable $x_i\in \R^{d_i}$, has access to the following local components in the global objective and constraints: 1) the local objective function $f_i:\R^{d_i}\rightarrow\R$, 2)~the local inequality constraint function $g_i=[g_{i1},\ldots,g_{im}]^T$, where $g_{ij}:\R^{d_i}\rightarrow\R$, 3) the local equality constraint function  $h_i:\R^{d_i}\rightarrow\R^p$,
and 4) the local constraint set $X_i\subseteq\R^{d_i}$. 

The problem solving process is required to be fully decentralized, i.e., each agent can only communicate with its neighbors in the neighborhood $\calN_i:=\{j\ |\ \{i,j\}\in\calE\}$. The challenge here is to handle the coupled inequality and equality constraints, which make the local decisions intertwined.

We attempt to exploit the underlying separability of problem~\eqref{prob:CC} by virtue of duality. The Lagrange dual of the above problem is given by 
\begin{gather}
    \underset{\mu\in\R^m_+, \lambda\in\R^p}{\text{minimize}}\quad -q(\mu, \lambda):=-\sum_{i\in\calV}q_i(\mu,\lambda), \tag{$\calD$}\label{prob:D}
\end{gather}
where $q_i(\mu,\lambda):=\inf_{x_i\in X_i}f_i(x_i) + \la\mu,g_i(x_i)\ra + \la\lambda,h_i(x_i)\ra$. 
Here, the dual function $q(\mu,\lam)$ is essentially nonsmooth. 

We impose the following standard assumptions on problem~\eqref{prob:CC}.
\begin{assumption} \label{asp:problem_structure}
Problem \eqref{prob:CC} satisfies the following:
    \begin{enumerate}
        \item For each $i\in\calV$, $X_i$ is a convex set.
        \item For each $i\in\calV$, $f_i$, $g_i$ are convex functions on $X_i$, and $h_i$ is an affine function on $X_i$. 
        \item There exists at least one optimal solution to \eqref{prob:CC}.
    \end{enumerate}
\end{assumption}

\begin{assumption}[Slater's condition] \label{asp:slater}
    There exist $\tilx_1,\tilx_2,\ldots,\linebreak\tilx_N$ such that $\tilx_i\in \mathrm{relint}X_i$ $\forall i\in\calV$, $\sum_{i\in\calV}g_i(\tilx_i)<\0_m$ and $\sum_{i\in\calV}h_i(\tilx_i) = \0_p$. 
\end{assumption}

It is important to note that we do not impose any smoothness assumption on $f_i$'s and $g_i$'s, neither require $X_i$'s to be compact. Thus this form of \eqref{prob:CC} is more general than those in \cite{MaC17, CNS14, HMS24, FMG17, WWL23, LWY21, LLS20, LWY20, AyH19, NeS18, SiJ16}. 

According to \cite[Proposition 5.3.5]{Ber09}, the above assumptions guarantee zero duality gap and the existence of a primal-dual optimal pair $(\x^\star, (\mu^\star, \lam^\star))\in\R^{\sum d_i}\times\R_+^{m}\times\R^p$ of the nonsmooth convex optimization problem~\eqref{prob:CC} satisfying
\begin{align} \label{eq:optimal_primal_dual_pair1}
    \begin{split}
        \x^\star \in \arg\min_{\x\in X} f(\x) + \la\mu^\star, (\1_N\otimes I_m)^Tg(\x)\ra \\
        + \la\lam^\star, (\1_N\otimes I_p)^Th(\x)\ra,
    \end{split}
\end{align} 
where $\x:=[x_1^T,x_2^T,\ldots,x_N^T]^T\in\R^{\sum d_i}$, $f(\x) := \sum_{i\in\calV} f_i(x_i)$, $g(\x):=[g_1(x_1)^T, \ldots, g_N(x_N)^T]^T$, $h(\x):=[h_1(x_1)^T, \ldots, h_N(x_N)^T]^T$, and $X:=X_1\times X_2\times\cdots\times X_N$ is the Cartesian product of the local constraint sets.

Subsequently, we derive an equivalent form of the dual problem~\eqref{prob:D}, which facilitates our algorithm development in Section~\ref{sec:development}. 
We first assign each node $i$ a local dual variable $y_i:=[\mu_i^T,\lam_i^T]$, where $\mu_i\in\mathbb{R}^m$ and $\lam_i\in\mathbb{R}^p$ are node $i$'s estimates of the global dual variables $\mu$ and $\lam$ in problem~\eqref{prob:D}, respectively, and consider
\begin{align}
    \label{prob:D_prime} \tag{$\calD'$}
    \begin{split}
        \underset{\y\in\calK}{\text{minimize}}\quad & -\tilq(\y) := -\sum_{i\in\calV}q_i(y_i) \\
        \text{subject to}\quad &\quad\ \tilH^{\half}\y = \0.
    \end{split}
\end{align}
Here, $\y:=[y_1^T,y_2^T,\ldots,y_N^T]^T$ and we separate $\y$ into two parts $\y_\mu:=[\mu_1^T,\mu_2^T,\ldots,\mu_N^T]^T\in\R^{Nm}$ and $\y_\lam:=[\lam_1^T,\lam_2^T,\ldots,\lam_N^T]^T\in\R^{Np}$. With the abuse of notation, we let $q_i(y_i):=q_i(\mu_i,\lam_i)$. Also, the matrix $\tilH\in\R^{N(m+p)\times N(m+p)}$ is a symmetric positive semidefinite matrix such that $\mathrm{Null}(\tilH)=\{\y\ |\ y_1=\cdots=y_N\}$, so that the equality constraint in \eqref{prob:D_prime} forces all the $y_i$'s to be identical. Moreover, the constraint set $\calK$ is given by $\calK:=\{\y\in\R^{N(m+p)}\ |\ \mu_i\in\R_+^m,\ \forall i\in\calV\}$, or equivalently, $\calK:=\{\y\in\R^{N(m+p)}\ |\ \y_\mu \in\R_+^{Nm}\}$. Consequently, problem~\eqref{prob:D_prime} is equivalent to \eqref{prob:D}.

Note that \eqref{prob:D_prime} is in the form of \emph{consensus optimization} \cite{SLY14}, i.e., finding a consensus that minimizes the sum of certain local functions associated with the nodes. Nevertheless, a large volume of methods for consensus optimization that dualize the consensus constraint are inapplicable (e.g., \cite{AWL18, SLW15b, MaO17,HoC17}), as they typically need to evaluate a proximal map of the nonsmooth component of the objective function at each iteration. In our case, this nonsmooth part is the dual function $\tilq$. However, obtaining an explicit form of the dual function is generally difficult. For instance, consider the simple problem:
\begin{align*}
    \underset{x\in\mathbb{R}^d}{\text{minimize}}\quad \norm{Ax-b}_1\quad\text{subject to}\quad Cx=d.
\end{align*}
The corresponding dual function is $\inf_x\norm{Ax-b}_1+\lambda^T(Cx-d)$, which is challenging to express explicitly. Consequently, directly evaluating its proximal map is infeasible. For the same reason, applying the methods mentioned above to \eqref{prob:D_prime} is not straightforward. Later this issue will be overcome by our algorithm design.

\section{Algorithm Development} \label{sec:development}

In this section, we present two novel methods for solving problem~\eqref{prob:CC} and \eqref{prob:D_prime}. The first method tackles the challenge of computing the proximal map for a dual function lacking an explicit expression. The second method is a proximal variant of the first one, which expands the solvable problem range.

\subsection{Review of Approximated Method of Multipliers (AMM)}

As is mentioned in the last section, existing consensus optimization methods are often inapplicable to problem~\eqref{prob:D_prime}. To overcome this issue, we introduce an indirect approach direct computation. 

Subsequently, we briefly introduce AMM \cite{WuL23}, which serves as the cornerstone of our approach. Applying AMM to \eqref{prob:D_prime} yields the following algorithmic form\footnote{We set the surrogate function $u^k(\y)$ in \cite{WuL23} as $\norm{\y-\y^k}_A^2$.}: starting from arbitrary $\y^0\in\calK$ and $\z^0\in\R^{N(m+p)}$, 
    \begin{align}
        \begin{split}
            \y^{k+1} &= \arg\min_{\y\in\calK} \Big\{ - \tilq(\y) + \la\tilH^\half\z^k, \y\ra \\
            &\qquad\qquad\qquad + \frac{\rho}{2}\norm{\y}_H^2 + \half\norm{\y-\y^k}_{A}^2 \Big\}
        \end{split} \label{eq:AMM_y_update}\\
        \z^{k+1} &= \z^k + \rho \tilH^\half\y^{k+1},\ \forall k\geq 0. \label{eq:AMM_z_update}
    \end{align}

The above form of AMM can be viewed as being induced from the proximal method of multipliers \cite{Roc76}, but has three weight matrices  $A=P_A\otimes I_{m+p}$, $H=P_H\otimes I_{m+p}$, and $\tilH=P_\tilH\otimes I_{m+p}$ to provide more flexibility, where $P_A, P_H, P_\tilH\in\R^{N\times N}$.
Later in Section~\ref{subsec:implementations}, we will further impose neighbor-spare sparse structures on these matrices to enable distributed implementations. For now, we assume some basic properties of them to aid our algorithm design.

\begin{assumption} \label{asp:matrices}
    The matrices $P_A, P_H$ and $P_\tilH$ are symmetric positive semidefinite and satisfy that $P_A+\rho P_H\succ 0$, $P_H\succeq P_\tilH$, and  $\mathrm{Null}(P_H)= \mathrm{Null}(P_\tilH) = \mathrm{span}(\1_N)$. 
\end{assumption}

\begin{remark}
    The above assumption ensures that 1) $H$, $\tilH$ and $A$ are symmetric positive semidefinite; 2) $A+\rho H\succ 0$, $H\succeq \tilH$, and $\mathrm{Null}(H)=\mathrm{Null}(\tilH)=S$, where 
    \begin{gather}
        S := \{ [y_1^T,\ldots, y_N^T]^T \in \R^{N(m+p)}\ |\ y_1 = \cdots = y_N \}.\label{eq:nullspace_S}
    \end{gather}
\end{remark}

Note that the updates \eqref{eq:AMM_y_update}--\eqref{eq:AMM_z_update} are not implementable, since the explicit expression of $\tilq$ is generally inaccessible.

\subsection{DUCA with Indirect Computation of \eqref{eq:AMM_y_update}}\label{subsec:DUCA}
In the sequel, we derive a realizable algorithmic form of \eqref{eq:AMM_y_update}-\eqref{eq:AMM_z_update}. 
To this end, notice that from the first-order optimality condition \cite[Proposition~5.4.7]{Ber09}, \eqref{eq:AMM_y_update} is equivalent to 

\begin{align}
    \label{eq:optimal_condition_derivation}
    \begin{split}
        \0 &\in -\partial \tilq(\y^{k+1}) + (A+\rho H)\y^{k+1}  \\
    & \qquad \qquad - A\y^k + \tilH^\half\z^k + \mathscr{N}_{\calK}(\y^{k+1}),
    \end{split}
\end{align}
where $\mathscr{N}_{\calK}(\y^{k+1})$ is the normal cone of $\calK$ at $\y^{k+1}$. 

\begin{assumption}\label{asp:compact}
    For each $i\in\calV$, $X_i$ is compact.
\end{assumption}

Assumption~\ref{asp:compact} implies that $X=X_1\times X_2\times\cdots\times X_N$ is also compact. Under this assumption, by applying Danskin's theorem \cite{Ber16} to 
\begin{equation}
    \tilq(\y^{k+1}) = \min_{\x\in X} f(\x) + \la\y^{k+1}, \tilg(\x)\ra, \label{eq:danskin}
\end{equation}
we obtain $-\partial\nobreak \tilq(\y^{k+1}) = \{-\tilg(\x) |\ \x\in \calX(\y^{k+1})\}$, where
$\tilg(\x)$ denotes $[(g_1(x_1))^T, (h_1(x_1))^T, \ldots, (g_N(x_N))^T, \linebreak(h_N(x_N))^T]^T$,
and $\calX(\y^{k+1})$ is the set of all the minimizing points in \eqref{eq:danskin}. This set is guaranteed to be nonempty and compact according to Weierstrass' theorem \cite{Ber16}, and convex due to the convexity of $f$, $g$ and $h$.

As a result, \eqref{eq:optimal_condition_derivation} is equivalent to the following statement: given $\y^{k+1}\in\R^{N(m+p)}$, there exist $\x^{k+1}\in \calX(\y^{k+1})$ and $-\bsig^{k+1}\in \mathscr{N}_\calK(\y^{k+1})$ such that
\begin{align}
    \begin{split}
        \y^{k+1} &= D^{-1}\big(A\y^k-\tilH^\half\z^k
         +\tilg(\x^{k+1})+\bsig^{k+1}\big),
    \end{split} \label{eq:y_update_derivation1}
\end{align}
where $D:=A+\rho H$. Note that $D=(P_A+\rho P_H)\otimes I_{m+p}$ and $D\succ 0$ by Assumption~\ref{asp:matrices}. To realize \eqref{eq:y_update_derivation1} in practice, we further impose a diagonal assumption on $D$.

\begin{assumption}\label{asp:diagonal}
    The matrix $P_A+\rho P_H$ is diagonal. 
\end{assumption}

The following lemma presents the indirect computation strategy and establishes our first algorithm.

\begin{lemma}\label{lem:compute_x}
    Suppose Assumption~\ref{asp:problem_structure}-\ref{asp:diagonal} hold. Given $\y^k$ and $\z^k$, then there exist variables $\y^{k+1}\in\R^{N(m+p)}$, $\x^{k+1}\in X$ and $\bsig^{k+1}\in\R^{N(m+p)}$ satisfying: $\x^{k+1}\in \calX(\y^{k+1})$, $-\bsig^{k+1}\in \mathscr{N}_\calK(\y^{k+1})$ and \eqref{eq:y_update_derivation1}. Moreover, such $\x^{k+1}$ and $\y^{k+1}$ can be sequentially computed by
    \begin{align}
        \begin{split}
            \x^{k+1} &\in \arg\min_{\x\in X} \bigg\{f(\x) + \frac{1}{2}\Big\Vert
            \calP_\calK\big[A\y^k-\tilH^\half\z^k\\ 
            &\qquad\qquad\qquad\qquad\qquad\qquad\qquad +\tilg(\x)\big]
        \Big\Vert_{D^{-1}}^2 \bigg\}, \label{eq:alg1_x_update} 
        \end{split}
        \\
        \y^{k+1} &= \calP_\calK\Big[D^{-1}\big(A\y^k-\tilH^\half\z^k+\tilg(\x^{k+1}) \big)\Big]. \label{eq:alg1_y_update} 
    \end{align}
\end{lemma}
\begin{proof}
    See Appendix~\ref{app:lem:compute_x}.
\end{proof}

This lemma, together with \eqref{eq:y_update_derivation1}, indicates that $\y^{k+1}$ given by \eqref{eq:AMM_y_update}, which cannot be directly computed, can now be computed via \eqref{eq:alg1_x_update} and \eqref{eq:alg1_y_update}. We refer to the algorithm described by \eqref{eq:alg1_x_update}, \eqref{eq:alg1_y_update} and \eqref{eq:AMM_z_update}, with initialization $\y^0\in\calK$ and $\z^0\in\R^{N(m+p)}$, as the unified DUal Consensus Algorithm (DUCA), which unifies various existing methods for consensus optimization.

\begin{remark}
    When Assumption~\ref{asp:problem_structure}-\ref{asp:diagonal} hold, the sequences $\{\y^k,\linebreak \z^k\}_{k\geq 0}$ generated by \eqref{eq:AMM_y_update} and \eqref{eq:AMM_z_update} are identical to the corresponding sequences generated by DUCA. This indirect approach allows methods generalized by AMM (originally developed for consensus optimization) to be applied to the dual problem of constrained convex optimization problems of the form \eqref{prob:D_prime}, even when there is no closed-form dual function. Examples of such algorithms include PGC \cite{HoC17}, P-EXTRA \cite{SLW15b}, DPGA \cite{AWL18}, distributed ADMM \cite{MaO17}, and Proximal-Tracking \cite{FaP22}. 
    Furthermore, our convergence analysis of DUCA in Section~\ref{sec:convergence} provides results for the primal recovery process of all these algorithms.
    
\end{remark}

\subsection{Pro-DUCA}

The compactness assumption on $X_i$'s, i.e., Assumption~\ref{asp:compact}, is necessary to ensure that \eqref{eq:alg1_x_update} has a solution. However, this could make our problem formulation restrictive. To eliminate this assumption, we modify DUCA by adding a proximal term to \eqref{eq:alg1_x_update} as follows:
\begin{align}
    \begin{split}
        \x^{k+1} &= \arg\min_{\x\in X} \Big\{f(\x) + \frac{1}{2}\norm{
            \calP_\calK\big[A\y^k-\tilH^\half\z^k\\
        &\qquad\qquad\qquad\ \ +\tilg(\x)\big]}_{D^{-1}}^2+ \frac{\alpha}{2}\norm{\x-\x^k}^2\Big\} \label{eq:alg2_x_update},
    \end{split}
\end{align}
where $\alpha>0$. This modification adds no extra computational cost. Instead, it makes the minimization problem in \eqref{eq:alg2_x_update} $\alpha$-strongly convex, yielding a unique solution $\x^{k+1}$ and enhances computational efficiency \cite{Roc76}. The proposed proximal algorithm described by \eqref{eq:alg2_x_update}, \eqref{eq:alg1_y_update}, \eqref{eq:AMM_z_update}, with initialization $\y^0\in\calK$ and $\z^0\in\R^{N(m+p)}$ is called Pro-DUCA.

\begin{remark}
    This modification makes Pro-DUCA no longer equivalent to \eqref{eq:AMM_y_update}-\eqref{eq:AMM_z_update}. Accordingly, its convergence analysis cannot leverage the existing results of AMM and its special cases in \cite{WuL23}. We will provide new analysis in Section~\ref{sec:convergence}.
\end{remark}

\begin{algorithm}[tbp]
    \caption{DUCA/Pro-DUCA: Single-Exchange}
    \label{alg:implementation1}
    \begin{algorithmic}[1]
        \REQUIRE Each agent $i$ selects arbitrarily $x_i^0\in\R^{d_i}$, $y_i^0\in\R^m_+\times\R^p$, $v_i^0=\0_{m+p}$, and sends $y_i^{0}$ to all $j\in\calN_i$.
        \FORALL{$k\geq 0$}
            \FOR{$i=1, 2,\ldots, N$ (agents compute in parallel)} 
                \STATE $\tily_i^k = d_i' y_i^k  - \rho\sum\limits_{j\in\calN_i\cup\{i\}}\calL_{ij}y_j^k-v_i^k$.
                \STATE $x_i^{k+1}\in\arg\min_{x_i\in X_i} L_i^k(x_i, \tily_i^k)$.
                \STATE $y_i^{k+1} = \frac{1}{d_i'}\begin{bmatrix}
                    [\tilde{\mu}_i+g_i(x_i^{k+1})]_+\\
                    \tilde{\lam}_i + h_i(x_i^{k+1})
                \end{bmatrix}$.
                \STATE Agent $i$ sends $y_i^{k+1}$ to all $j\in\calN_i$.
                \STATE $v_i^{k+1} = v_i^k + \rho\sum\limits_{j\in\calN_i\cup\{i\}}\calL_{ij}y_j^{k+1}$.
            \ENDFOR
        \ENDFOR
    \end{algorithmic}
\end{algorithm}

\subsection{Distributed Implementations}\label{subsec:implementations}

We propose two distributed implementations for DUCA and Pro-DUCA with different parameter settings.

Consider any symmetric weight matrix $W=\{w_{ij}\}\in\R_+^{N\times N}$ such that
\begin{align}
    w_{ij} \begin{cases}
        > 0, & \text{ if } \{i,j\}\in\calE, \\ 
        \geq 0, & \text{ if } i=j,  \\
        0, & \text{otherwise},
    \end{cases}    \label{eq:penalty_mat}
\end{align}
and a diagonal matrix $\Lambda=\diag(\ell_1,\ldots,\ell_N)\in\R^{N\times N}$, where $\ell_i=\sum_{j\in\calV}w_{ij}\ \forall i\in\calV$. The graph Laplacian is defined as $\calL=\Lambda-W$, which is positive semidefinite, and has nullspace equal to $\mathrm{span}(\1_N)$ since the graph is connected \cite{GoR01}.

\subsubsection{Single-Exchange Implementation}

We set the algorithm parameters as follows: $P_H=P_\tilH=\calL$, and $P_D=\diag(d_1',\linebreak \ldots, d_N')$ is chosen to satisfy that $P_D\succ 0, P_A = P_D-\rho P_H\succeq 0$. 
\footnote{An option is to let $P_D=c \rho\Lambda$ with $c\geq 2$. Then $P_A=\rho(c-2)\Lambda+\rho\Lambda(I+\Lambda^{-1}W)$ will be positive semidefinite, because $\Lambda^{-1}W$ is row-stochastic and $\abs{\lam_{1}(\Lambda^{-1}W)}\leq 1$.} 
With the change of variable $\v^k := \tilH^\half\z^k$, the update \eqref{eq:AMM_z_update} of $\z^k$ becomes
\begin{gather} \label{eq:v_update}
    \v^{k+1} = \v^k + \rho\tilH\y^{k+1},
\end{gather}
and initially $\v^0\in\mathrm{range}(\tilH^\half)$. Here we simply choose $\v^0 = \0_{N(m+p)}$. With such setting, the first distributed implementation is proposed as Algorithm~\ref{alg:implementation1}. Here $\tily_i^k=[(\tilde{\mu}_i^k)^T, (\tilde{\lam}_i^k)^T]^T\in\R^{m+p}$ is an intermediate variable that equals to the $i$th component of $A\y^k-\v^k$, and $L_i^k$ is defined as
\begin{multline}
    L_i^k(x_i,y_i) = f_i(x_i) + \frac{1}{2d_i'}\biggl(\norm{[\mu_i+g_i(x_i)]_+}^2 \\+ \norm{\lam_i+h_i(x_i)}^2\biggr) + \frac{\alpha}{2}\norm{x_i-x_i^k}^2. \label{eq:alg2_AL_impl}
\end{multline}
In the formula above, we set $\alpha=0$ to implement DUCA and $\alpha>0$ to implement Pro-DUCA. 

This parameter setting corresponds to PGC \cite{HoC17}, P-EXTRA \cite{SLW15a} and DPGA \cite{AWL18}, with specific choices of $W$ and $P_D$. See \cite[Section~III-B]{WuL23} and Table~\ref{tab:param} for more details.

\subsubsection{Double-Exchange Implementation} \label{subsec:double-exchange}

We set $P_H=\calL\calM$, $P_\tilH=\calL^2$, and $P_D=\diag(d_1',\ldots, d_N')$ such that $P_D\succ 0$ and $P_A=P_D-\rho \calL\calM\succeq 0$. Here $\calM$ equals to $\calL$ (corresponding to distributed ADMM \cite{MaO17}
\footnote{Distributed ADMM \cite{MaO17} allows $W$ to be asymmetric. In that case, we set $P_H=P_\tilH=\calL^T\calL$. We assume symmetry for notation simplicity.})
or $2I-\calL$ (corresponding to Proximal-Tracking \cite{FaP22}, Augmented Lagrangian Tracking \cite{FaP23} and Tracking-ADMM \cite{FNN20}). In the later case, we need to further impose that i) $W\succeq 0$, and ii) $W$ is doubly stochastic.\footnote{\cite{FNN20,FaP22,FaP23} 
only require $W$ to be doubly stochastic, but do not require $W\succeq 0$. But if we have a doubly stochastic $W$, it is easy to make it additionally positive semidefinite: simply $\frac{I+W}{2}$ would suffice. In these cases, $P_D=\rho I$ would make Assumption~\ref{asp:matrices} holds.}
See \cite[Section~III-C]{WuL23} and Table~\ref{tab:param} for more details. 

In such setting, the second implementation is proposed as Algorithm~\ref{alg:implementation2}. Here $\tily_i^k=[(\tilde{\mu}_i^k)^T, (\tilde{\lam}_i^k)^T]^T\in\R^{m+p}$ still represents the $i$th component of $A\y^k-\tilH^\half\z^k = D\y^k - (\calL\otimes I_{m+p})\u^k$, where $\u^k=\rho(\calM\otimes I_{m+p})\y^k+\z^k$. $L_i^k$ is still defined as \eqref{eq:alg2_AL_impl}. Note that this implementation necessitates exchanging two distinct variables in each iteration, namely $y_i^k$ in Line~6, and $u_i^k$ in Line~9. Rooted inherently in all the methods it generalizes, which likewise require an equivalent exchange count, this double-exchange is an essential requirement.

\begin{algorithm}[tbp]
    \caption{DUCA/Pro-DUCA: Double-Exchange}
    \label{alg:implementation2}
    \begin{algorithmic}[1]
        \REQUIRE Each agent $i$ select arbitrarily $x_i^0\in\R^{d_i}$, $y_i^0\in\R^m_+\times\R^p$, $z_i^0=\0_{m+p}$, and sends $y_i^{0}$ to all $j\in\calN_i$.\\
        Then it computes $u_i^0=z_i^0 + \rho\sum\limits_{j\in\calN\cup\{i\}}\calM_{ij}y_j^{0}$, and sends $u_i^{0}$ to all $j\in\calN_i$.
        \FORALL{$k\geq 0$}
            \FOR{$i=1, 2,\ldots, N$ (agents compute in parallel)}
                \STATE $\tily_i^k = d_i' y_i^k - \sum\limits_{j\in\calN\cup\{i\}}\calL_{ij}u_j^k$.
                \STATE $x_i^{k+1}\in\arg\min_{x_i\in X_i} L_i^k(x_i, \tily_i^k)$.
                \STATE $y_i^{k+1} = \frac{1}{d_i'}\begin{bmatrix}
                    [(\tilde{\mu}_i^k+g_i(x_i^{k+1})]_+\\
                    \tilde{\lam}_i^k + h_i(x_i^{k+1})
                \end{bmatrix}$.
                \STATE Agent $i$ sends $y_i^{k+1}$ to all $j\in\calN_i$.
                \STATE $z_i^{k+1} = z_i^k + \rho\sum\limits_{j\in\calN\cup\{i\}}\calL_{ij}y_j^{k+1}$.
                \STATE $u_i^{k+1} = z_i^k + \rho\sum\limits_{j\in\calN\cup\{i\}}\calM_{ij}y_j^{k+1}$.
                \STATE Agent $i$ sends $u_i^{k+1}$ to all $j\in\calN_i$.
            \ENDFOR
        \ENDFOR
    \end{algorithmic}
\end{algorithm}

\section{Connections to the Existing Works}\label{sec:related}

In this section, we establish the connections between the DUCA/Pro-DUCA and existing methods, showing that DUCA generalizes \cite{HoC17,SLW15a,AWL18,MaO17,FNN20,FaP22,FaP23}, and Pro-DUCA generalizes special cases of \cite{GoZ23,WWL23}.

\subsection{Methods for Consensus Optimization}

If Assumption~\ref{asp:problem_structure}-\ref{asp:diagonal} hold, DUCA described by \eqref{eq:alg1_x_update}, \eqref{eq:alg1_y_update} and \eqref{eq:AMM_z_update} can be viewed as AMM \cite{WuL23} applied to the dual problem~\eqref{prob:D_prime}. With the indirect computation proposed in Section~\ref{subsec:DUCA}, many methods originally proposed for consensus optimization that are special cases of AMM become implementable to solve \eqref{prob:D_prime}, e.g., PGC over static networks \cite{HoC17}, P-EXTRA \cite{SLW15a}, DPGA \cite{AWL18}, and distributed ADMM \cite{MaO17}.

Proximal-Tracking \cite{FaP22} is a recently proposed method for consensus optimization. It can be viewed as the proximal counterpart of DIGing \cite{NOA17}, and is able to handle nonsmooth objectives. Using the same parameter setting as DIGing (see \cite[Section~III-C]{WuL23} and Section~\ref{subsec:double-exchange}), it is not hard to show that Proximal-Tracking is also a special case of AMM. Thus DUCA also extends its applications to problem~\eqref{prob:D_prime}.

Note that the above methods solves the dual problem~\eqref{prob:D_prime} rather than the primal problem~\eqref{prob:CC}, and all their convergence results are established for the dual sequences $\{\y^k, \z^k\}_{k\geq 0}$, dual objective error, and constraint violation in \eqref{prob:D_prime}, i.e., consensus error of $\y^k$. There has been no discussion of primal objective errors or constraint violations in \eqref{prob:CC}. To this end, we will establish convergence analysis of DUCA concerning the primal aspects, thereby contributing new insights into the primal recovery process for these methods.

\subsection{Methods for Problem~\eqref{prob:CC}}
We find that some of the existing methods for solving problem~\eqref{prob:CC} are also special cases of DUCA or Pro-DUCA, while some others are closely related to them.

\subsubsection{Augmented Lagrangian Tracking \cite{FaP23} and Tracking-ADMM \cite{FNN20}} These two methods can be viewed as specializations of DUCA, as is verified below.

Augmented Lagrangian Tracking has the following iterations: starting form $\x^0\in\R^{\sum d_i}$, $\y^0\in\R^{N(m+p)}$, $\bsig^0\in-\calK^\circ$, $\t^0 = -\tilg(\x^0)-\bsig^0$, then for $k\geq 0$, 
\begin{align}
    \begin{split}
        \x^{k+1} &= \arg\min_{x\in X} f(\x) + \frac{\rho}{2}\big\Vert\calP_\calK[W\y^k + \frac{1}{\rho}(\tilg(\x)\\
        &\qquad\qquad\qquad\qquad\qquad\quad-\tilg(\x^k)-W\t^k)]\big\Vert^2,
    \end{split} \allowdisplaybreaks\\
    \bsig^{k+1} &= -\calP_{\calK^\circ}[W\t^k-\tilg(\x^{k+1})+\tilg(\x^k)+\bsig^k-\rho W\y^k], \allowdisplaybreaks\\
    \t^{k+1} &= W\t^k - (\tilg(\x^{k+1})+\bsig^{k+1}) + (\tilg(\x^k)+\bsig^k), \label{eq:ALT_t_update} \allowdisplaybreaks\\
    \y^{k+1} &= W\y^k - \frac{1}{\rho}\t^{k+1}. \label{eq:ALT_y_update}
\end{align}

It can be shown from \eqref{eq:ALT_y_update} that $\y^{k+1}-W\y^k=W\y^k-W^2\y^{k-1}-\frac{1}{\rho}(\t^{k+1}-W\t^k)$. Using \eqref{eq:ALT_t_update} to eliminate $\t$ in this relation, we obtain the dynamic of $\y$: $\y^{k+1} = 2W\y^k-W^2\y^{k-1}+\frac{1}{\rho}\big((\tilg(\x^{k+1})+\bsig^{k+1}) - (\tilg(\x^k)+\bsig^{k})\big)$. Notice from \cite[Theorem~1]{FaP23} and \cite[Section~3.1]{FaP22} that $-(\tilg(\x^k)+\bsig^{k})$ is a subgradient of $-\tilde{q}(\cdot)+\delta_\calK(\cdot)$ at $\y^{k}$, which is consistent with the algorithm development of DUCA. Similarly, using \eqref{eq:AMM_z_update} to eliminate $\z$ in two successive iterations of \eqref{eq:relation_iterates}, one could verify that DUCA gives the same dynamic of $\y$, with $A=\rho W^2$, $D=\rho I$, $H=I-W^2$ and $\tilH=(I-W)^2$.

It is notable that Augmented Lagrangian Tracking can be viewed as applying Proximal-Tracking \cite{FaP22} to the dual problem \eqref{prob:D_prime}, which is consistent with our previous discussion that AMM generalizes Proximal-Tracking. 

As is pointed out in \cite[Section~3.4]{FaP23}, Tracking-ADMM \cite{FNN20} is a special case of Augmented Lagrangian Tracking, if there are no coupled inequality constraints in \eqref{prob:CC}. Thus it is also a special case of DUCA. In the next section, our analysis will enhance the convergence results of both of these works, as they only provide an asymptotic convergence without rates.

\subsubsection{IPLUX \cite{WWL23}}
The part of IPLUX dealing with coupled equality constraints can be viewed as ``P-EXTRA applied to the dual'' \cite[Section~III, Case~2]{WWL23}. In this case, IPLUX can be viewed as a specialization of Pro-DUCA, as it has also adds an proximal term in the primal update.

On the other hand, IPLUX deals with coupled inequality constraints by introducing an auxiliary variable, and transforming the constraints into coupled equality constraints and local inequality constraints. The local inequality constraints are further dealt by a virtual-queue-based algorithm \cite{YuN17}, which demands Lipschitz continuity for these constraints. In contrast, Pro-DUCA eliminates the need for such requirement and auxiliary variables, resulting in a simpler subproblem per iteration with fewer decision variables.

\subsubsection{DPMM \cite{GoZ23}} This method forms a Karush-Kuhn-Tucker (KKT) system of \eqref{prob:CC}, which is treated as an inclusion problem of a maximal monotone operator, and then uses ideas from the variable metric proximal point method and the prediction-correction framework to solve it.

Interestingly, although derived from different ideas, we find the algorithm structure of DPMM very similar to Pro-DUCA. However, there are some key differences:

a) Pro-DUCA allows for a broader choices of parameters and thus has stronger generalization ability. Specifically, the two weight matrices $H$ and $\tilH$ of Pro-DUCA are possibly different, and are not required to be compatible with the communication graph, while DPMM uses only one communication matrix that is compatible with the graph \cite[Assumption~1]{GoZ23}.

b) DPMM maintains a different sequence of primal variables $\{\x^k\}_{k\geq 0}$. Each time DPMM obtains $\hat{\x}^{k}$ from solving a subproblem like \eqref{eq:alg1_x_update}, it updates the primal variable $\x^{k+1}$ by $\x^{k+1} = (I-\Theta)\x^k + \Theta\hat{\x}^k$, where $\Theta$ is a diagonal matrix with each entry belonging to $(0,2)$. If $\Theta$ is set as $I$, then DPMM is special case of Pro-DUCA.

c) The convergence results are different. We will show in the next section that Pro-DUCA reaches $O(1/k)$ convergence rates in terms of both primal objective and feasibility errors; while DPMM achieves an $o(1/k)$ rate with respect to a first-order optimality residual. Moreover, DPMM enjoys a linear convergence rate when the problem satisfies more restrictive assumptions.

\section{Convergence Analysis}\label{sec:convergence}

In this section, we establishes the convergences results of DUCA and Pro-DUCA.

The Lagrange dual problem of \eqref{prob:D_prime} is 
\begin{gather}
    \underset{\z\in\R^{N(m+p)}}{\text{maximize}}\quad \big(\inf_{\y\in\calK}-\tilq(\y)+\la\z, \tilH^\half\y\ra\big). \label{prob:dual_dual}
\end{gather}
In the proposition below, we show that there is no duality gap, and provide a primal-dual optimal solution pair $(\y^\star,\z^\star)$ (we view $\y$ as the primal variable of \eqref{prob:D_prime} and $\z$ the dual variable).

\begin{proposition} \label{prop:z_star}
    Suppose that Assumption~\ref{asp:problem_structure}-\ref{asp:matrices} hold, $\x^\star, \mu^\star$ and $\lam^\star$ are defined in \eqref{eq:optimal_primal_dual_pair1}. Then there is no duality gap between \eqref{prob:D_prime} and \eqref{prob:dual_dual}. Furthermore, $(\y^\star, \z^\star)$ is a primal-dual optimal solution pair, where $\y^\star=\1_N \otimes [(\mu^\star)^T,(\lam^\star)^T]^T$, $\z^\star=(\tilH^\half)^\dag\tilg(\x^\star)$.
\end{proposition}
\begin{proof}
    See Appendix~\ref{app:prop:z_star}.
\end{proof}

In the following analysis, we use $\barby^k:= \frac{1}{k}\sum_{l=1}^k\y^l$ and $\barbx^k:= \frac{1}{k}\sum_{l=1}^k \x^l$ to denote the iterate averages. We also use auxiliary variables $\v$ and $\bsig$: $\v^\star:=\tilH^\half\z^\star$,  $\v^k:=\tilH^\half\z^k\ \forall k\geq 0$ with its update formula~\eqref{eq:v_update}, and 
\begin{align}
    \bsig^{k+1} &:= -\calP_{\calK^\circ}\Big[A\y^k - \tilH^\half\z^k +\tilg(\x^{k+1})\Big]\ \forall k\geq 0,\label{eq:sig_update}
\end{align}
where $\calK^\circ\subset\R^{N(m+p)}$ is the polar cone of $\calK$. We will show in the proof of Proposition~\ref{prop:relations} that the definition of $\bsig^k$ is consistent with the one used in \eqref{eq:y_update_derivation1}.

In view that DUCA and Pro-DUCA share two identical update formulae, namely \eqref{eq:alg1_y_update} and \eqref{eq:AMM_z_update}, the iterates generated by them naturally share some key properties.

\begin{proposition} \label{prop:relations}
    Suppose that Assumption~\ref{asp:problem_structure}-\ref{asp:diagonal} hold (Pro-DUCA does not need Assumption~\ref{asp:compact}). The sequences $\{\x^k, \y^k, \z^k\}_{k\geq 0}$ generated by DUCA or Pro-DUCA satisfy the following.
    \begin{enumerate}
        \item (recursive relation of iterates) For all $k\geq 0$, 
        \begin{align}
            D\y^{k+1} = A\y^k - \tilH^\half\z^k + \tilg(\x^{k+1}) + \bsig^{k+1}. \label{eq:relation_iterates}
        \end{align} 
        
        \item (relation between accumulative constraint violation and dual iterates) For all $k\geq 1$,
        \begin{align}
            &(\1_N\otimes I_{m+p})^T\sum_{l=1}^{k}\big(\tilg(\x^l)+\bsig^l\big) \nonumber\\
            &\qquad\qquad\qquad= (\1_N\otimes I_{m+p})^T A(\y^k-\y^0).\label{eq:relation_y_A_constraints}
        \end{align}

        \item (ergodic constraint violation is bounded by dual iterates) For all $k\geq 1$, 
        \begin{align}
            \left\Vert\begin{bmatrix}
                \big[\sum_{i=1}^N g_i(\barx_i^k)\big]_+ \\
                \sum_{i=1}^N h_i(\barx_i^k)
            \end{bmatrix}\right\Vert
            \leq \frac{\sqrt{N\lam_{1}(P_A)}}{k} \norm{\y^k-\y^0}_A. \label{eq:cons_vio_bounded_by_y_A_norm}
        \end{align}

        \item (ergodic objective error is lower-bounded by ergodic constraint violation) For all $k\geq 1$, 
        \begin{align}
            f(\barbx^k) - f(\x^\star )\geq 
            - \left\Vert y^\star \right\Vert  
            \left\Vert\begin{bmatrix}
                \big[\sum_{i=1}^N g_i(\barx_i^k)\big]_+\\
                \sum_{i=1}^N h_i(\barx_i^k)
            \end{bmatrix} \right\Vert .
            \label{eq:oe_bounded_by_fe}
        \end{align}

        \item For all $k\geq 0$,
        \begin{multline}
            \la\y^{k+1}, D(\hatby^k - \y^{k+1}) +\v^\star\ra \leq \frac{1}{2}(\norm{\y^k}_A^2 - \norm{\y^{k+1}}_A^2) \\
            + \frac{1}{2\rho}(\norm{\v^k-\v^\star}_{\tilH^\dag}^2 - \norm{\v^{k+1}-\v^\star}_{\tilH^\dag}^2), \label{eq:relation_v_star_lyapunov}
        \end{multline}
        where we denote $\hatby^k:=D^{-1}(A\y^k-\tilH^\half\z^k)$.
    \end{enumerate}
\end{proposition}
\begin{proof}
    See Appendix~\ref{app:prop:relations}
\end{proof}

\subsection{Convergence Analysis of DUCA}
Under Assumption~\ref{asp:problem_structure}-\ref{asp:diagonal}, DUCA (\eqref{eq:alg1_x_update}, \eqref{eq:alg1_y_update} and \eqref{eq:AMM_z_update}) is equivalent to AMM applied to \eqref{prob:D_prime}, thus we can leverage on some existing results of AMM \cite{WuL23}: the sequences $\{\norm{\y^k}_A\}$ and $\{\norm{\z^k}\}$ are bounded; both of the objective error and consensus error of problem~\eqref{prob:D_prime} at the averaged iterate $\barby^k$ converge at an $\bigO(\frac{1}{k})$ rate. We now quote these results in our notation.

Let $\s^k:=[(\y^k)^T,(\z^k)^T]^T$, $\s^\star:=[(\y^\star)^T,(\z^\star)^T]^T$ and $G:=\mathrm{diag}(A, I_{N(m+p)}/\rho)$.

\begin{lemma}[{\cite[Lemma~3]{WuL23}}]\footnote{To identify \eqref{eq:lem:bounded_z_y_y}, we need to take $\Lambda_M=\0_{N\times N}$ into the proof of \cite[Lemma~3]{WuL23}.} \label{lem:bounded_z_y}
    Suppose that Assumption~\ref{asp:problem_structure}-\ref{asp:diagonal} hold. The sequences $\{\x^k\}_{k\geq 1}, \{\y^k\}_{k\geq 0}$  and $\{\z^k\}_{k\geq 0}$ generated by DUCA satisfies that for all $k\geq 1$,
    \begin{gather}
        \norm{\z^k-\z^0} \leq \norm{\z^0-\z^\star} + \sqrt{\rho}\norm{\s^0-\s^\star}_G, \label{eq:lem:bounded_z_y_z}\\
        \norm{\y^k-\y^0}_A \leq \norm{\y^0-\y^\star}_A + \norm{\s^0-\s^\star}_G. \label{eq:lem:bounded_z_y_y}
    \end{gather}
\end{lemma}

\begin{theorem}[{\cite[Theorem~1]{WuL23}}]
    Suppose Assumption~\ref{asp:problem_structure}-\ref{asp:diagonal} hold. The sequences $\{\x^k\}_{k\geq 1}, \{\y^k\}_{k\geq 0}$, and $\{\z^k\}_{k\geq 0}$ generated by DUCA satisfy that for all $k\geq 1$,
    \begin{align*}
        \norm{\tilH^{\half}\barby^k}&\leq \frac{1}{\rho k}\left(\norm{\z^0-\z^\star} + \sqrt{\rho} \norm{\s^0-\s^\star}_G \right), \\
        \begin{split}
            \tilq(\y^\star) - \tilq(\barby^k) &\leq \frac{1}{2k}\big(\frac{\norm{\z^0}^2}{\rho} + \norm{\y^0-\y^\star}_A^2 + \norm{\s^0-\s^\star}_G^2 \big),
        \end{split}\\
        \tilq(\y^\star) - \tilq(\barby^k) &\geq -\frac{\norm{\z^\star}}{\rho k}{\left(\norm{\z^0-\z^\star}+\sqrt{\rho}\norm{\s^0-\s^\star}_G\right)}.
    \end{align*}
\end{theorem}


However, to provide convergence results in terms of the primal problem~\eqref{prob:CC}, we need new analysis. The following theorem provides an $O(1/k)$ convergence rate of the primal feasibility error at the averaged iterate $\barbx^k$, which is essentially guaranteed by the boundedness of dual iterates.

\begin{theorem}[feasibility error] \label{thm:alg1_feasibility_error}
    Suppose that Assumption~\ref{asp:problem_structure}-\ref{asp:diagonal} hold. The sequences $\{\x^k\}_{k\geq 1}, \{\y^k\}_{k\geq 0}$ and $\{\z^k\}_{k\geq 0}$ generated by DUCA satisfy that for all $k\geq 1$, 
    \begin{multline}
        \left\Vert\begin{bmatrix}
            \big[\sum_{i=1}^N g_i(\barx_i^k)\big]_+\\
            \sum_{i=1}^N h_i(\barx_i^k)
        \end{bmatrix}\right\Vert 
        \leq \frac{\sqrt{N\lam_{1}(P_A)}}{k} \big(\norm{\y^0-\y^\star}_A \\+ \norm{\s^0-\s^\star}_G\big). \label{eq:thm:alg1_fe_result}
    \end{multline}
\end{theorem}
\begin{proof}
    Combining \eqref{eq:cons_vio_bounded_by_y_A_norm} with \eqref{eq:lem:bounded_z_y_y}, the result follows.
\end{proof}


To show the convergence rate of objective error, we need the following lemma.

\begin{lemma} \label{lem:alg1_objective_error1}
    Suppose that Assumption~\ref{asp:problem_structure}-\ref{asp:diagonal} hold and the sequences $\{\x^k\}_{k\geq 1}, \{\y^k\}_{k\geq 0}, \{\z^k\}_{k\geq 0}$ are generated by DUCA. Then for all $k\geq 0$,
    \begin{multline}
        f(\x^{k+1})-f(\x^\star) \leq \frac{1}{2}(\norm{\y^k}_A^2 - \norm{\y^{k+1}}_A^2)\\ 
        + \frac{1}{2\rho}(\norm{\v^k-\v^\star}_{\tilH^\dag}^2
        - \norm{\v^{k+1}-\v^\star}_{\tilH^\dag}^2).
    \end{multline} 
\end{lemma}

\begin{proof} 
    See Appendix~\ref{app:lem:alg1_objective_error1}.
\end{proof}

\begin{theorem}[objective error]
    Suppose that Assumption~\ref{asp:problem_structure}-\ref{asp:diagonal} hold. The sequences $\{\x^k\}_{k\geq 1}, \{\y^k\}_{k\geq 0}, \{\z^k\}_{k\geq 0}$ are generated by DUCA. Then for any $k\geq 1$,
    \begin{gather}
        -\frac{R_1}{k} \leq f(\barbx^k)-f(\x^\star) \leq \frac{R_2}{k}, \label{eq:thm:alg1_oe_result}
    \end{gather}
    where $R_1:=\norm{y^\star}\sqrt{N\lam_{1}(P_A)} \big(\norm{\y^0-\y^\star}_A + \norm{\s^0-\s^\star}_G\big)$ and $R_2:=\frac{1}{2\rho}\norm{\v^0-\v^\star}_{\tilH^\dag}^2+\half\norm{\y^0}_A^2$.
\end{theorem}
\begin{proof} 
    The first half of \eqref{eq:thm:alg1_oe_result} is due to \eqref{eq:cons_vio_bounded_by_y_A_norm} and Theorem~\ref{thm:alg1_feasibility_error}. The second half is proved by considering: 
    $
         f(\barbx^k)-f(\x^\star) \leq \frac{1}{k}\sum_{l=1}^k\big(f(\x^l)-f(\x^\star)) 
        \leq{} \frac{1}{k}\big(\frac{1}{2\rho}\norm{\v^0-\v^\star}_{\tilH^\dag}^2+\half\norm{\y^0}_A^2\big),
    $
    where the first inequality is by convexity of $f$ and the latter is due to Lemma~\ref{lem:alg1_objective_error1}.
\end{proof}

\subsection{Convergence Analysis of Pro-DUCA}

Pro-DUCA is no longer equivalent to AMM applied to problem~\eqref{prob:D_prime}, since a proximal term is added to the update of $\x^k$. For the this reason, the convergence analysis needs to be built from scratch; and for the same reason, it can solve more general problems with unbounded constraint set $X$, i.e., Assumption~\ref{asp:compact} will be discarded in the following analysis.

\begin{lemma}\label{lem:alg2_objective_error}
    Suppose that Assumption~\ref{asp:problem_structure}-\ref{asp:matrices} and Assumption~\ref{asp:diagonal} hold and the sequences $\{\x^k, \y^k, \z^k\}_{k\geq 0}$ are generated by Pro-DUCA. Then for all $k\geq 0$,
    \begin{multline}
        f(\x^{k+1})-f(\x^\star) \leq \frac{1}{2\rho}( \norm{\v^k-\v^\star}_{\tilH^\dag}^2 - \norm{\v^{k+1}-\v^\star}_{\tilH^\dag}^2) \\
        + \frac{1}{2}(\norm{\y^k}_A^2 - \norm{\y^{k+1}}_A^2)
         + \frac{\alpha}{2}(\norm{\x^k-\x^\star}^2 - \norm{\x^{k+1}-\x^\star}^2).
    \end{multline}
\end{lemma}
\begin{proof}
    See Appendix~\ref{app:lem:alg2_objective_error}.
\end{proof}

Using the lemma above, the boundedness of $\{\norm{\y^k}_A\}$, and the convergence rates of feasibility and objective errors, can all be proved. The proofs are inspired by \cite{WWL23}.

\begin{lemma}\label{lem:alg2_y_bounded}
    Suppose that Assumption~\ref{asp:problem_structure}-\ref{asp:matrices} and Assumption~\ref{asp:diagonal} hold and the sequences $\{\x^k, \y^k, \z^k\}_{k\geq 0}$ are generated by Pro-DUCA. Then for all $k\geq 1$,
    \begin{gather}
        \norm{\y^k}_A\leq C_1 + C_2,\label{eq:lem:alg2_y_bounded_result}
    \end{gather}
    where $C_1=\sqrt{N\lam_{1}(P_A)}\norm{y^\star}$, $C_2=((\norm{\y^0}_A+C_1)^2 + \alpha\norm{\x^0-\x^\star}^2 + \frac{1}{\rho}\norm{\v^0-\v^\star}_{\tilH^\dag}^2)^\half$.
\end{lemma}
\begin{proof}
    See Appendix~\ref{app:lem:alg2_y_bounded}.
\end{proof}

\begin{theorem}[feasibility error]\label{thm:alg2_feasibility_error}
    Suppose that Assumption~\ref{asp:problem_structure}-\ref{asp:matrices} and Assumption~\ref{asp:diagonal} hold and the sequences $\{\x^k, \y^k, \z^k\}_{k\geq 0}$ are generated by Pro-DUCA. Then for all $k\geq 1$, 
    \begin{align}
        \left\Vert\begin{bmatrix}
            \big[\sum_{i=1}^N g_i(\barx_i^k)\big]_+\\
            \sum_{i=1}^N h_i(\barx_i^k)
        \end{bmatrix}\right\Vert 
        \leq \frac{\sqrt{N\lam_{1}(P_A)}}{k} \Big(\norm{\y^0}_A + C_1 + C_2\Big),\label{eq:thm:alg2_fe_result}
    \end{align}
    where $C_1$ and $C_2$ are defined in same way as in Lemma~\ref{lem:alg2_y_bounded}.
\end{theorem}
\begin{proof}
    Combining \eqref{eq:cons_vio_bounded_by_y_A_norm} with Lemma~\ref{lem:alg2_y_bounded}, the result follows.
\end{proof}

\begin{theorem}[objective error]\label{thm:alg2_objective_error}
    Suppose that Assumption~\ref{asp:problem_structure}-\ref{asp:matrices} and Assumption~\ref{asp:diagonal} hold and the sequences $\{\x^k, \y^k, \z^k\}_{k\geq 0}$ are generated by Pro-DUCA. Then for all $k\geq 1$, 
    \begin{gather}
        -\frac{R_1^\prime}{k} \leq f(\barbx^k)-f(\x^\star) \leq \frac{R_2^\prime}{k}, \label{eq:thm:alg2_oe_result}
    \end{gather}
    where $R_1^\prime=C_1(\norm{\y^0}_A + C_1 + C_2)$, $R_2^\prime = \frac{1}{2\rho}\norm{\v^0-\v^\star}_{\tilH^\dag}^2+\half\norm{\y^0}_A^2+\frac{\alpha}{2}\norm{\x^0-\x^\star}^2$, $C_1$ and $C_2$ are defined in same way as in Lemma~\ref{lem:alg2_y_bounded}.
\end{theorem}
\begin{proof}
    The first half of \eqref{eq:thm:alg1_oe_result} is due to \eqref{eq:oe_bounded_by_fe} and Theorem~\ref{thm:alg2_feasibility_error}.
    The second half is proved by considering
    $
        f(\barbx^k)-f(\x^\star) \leq \frac{1}{k}\sum_{l=1}^k\big(f(\x^l)-f(\x^\star))
        \leq{} \frac{1}{k}\big(\frac{1}{2\rho}\norm{\v^0-\v^\star}_{\tilH^\dag}^2+\half\norm{\y^0}_A^2+\frac{\alpha}{2}\norm{\x^0-\x^\star}^2\big),
    $
    where the first inequality is by convexity of $f$ and the latter is due to Lemma~\ref{lem:alg2_objective_error}.
\end{proof}

\begin{remark}\label{rem:weight_matrices_effect}
    Now we analyze the effect of the weight matrices. Taking $\y^0=\z^0=\0$ and $\x^0=\0$, it is not hard to see that the constant terms in the feasibility error in \eqref{eq:thm:alg1_fe_result} and \eqref{eq:thm:alg2_fe_result} are upper-bounded by $2N\lam_1(P_A)\norm{y^\star}
    + \frac{\norm{\tilg(\x^\star)}}{\rho\lam_{N-1}(P_\tilH)}$ and $N\lam_1(P_A)\norm{y^\star} 
    + \sqrt{\frac{N\lam_1(P_A)}{\rho\lam_{N-1}(P_\tilH)}}\norm{\tilg(\x^\star)}
    + \sqrt{N\lam_1(P_A)}(\half+\sqrt{\alpha}\norm{\x^*})$ respectively. In terms of objective error, the constant terms on the right-hand sides of \eqref{eq:thm:alg1_oe_result} and \eqref{eq:thm:alg2_oe_result} are upper-bounded by $\frac{\norm{\tilg(\x^\star)}^2}{2\rho\lam_{N-1}(P_\tilH)}$ and $\frac{\norm{\tilg(\x^\star)}^2}{2\rho\lam_{N-1}(P_\tilH)}+\frac{\alpha\norm{\x^\star}^2}{2}$ respectively. Therefore, guided by these upper bounds, when selecting parameters, smaller $\lam_1(P_A)$ and larger $\lam_{N-1}(P_\tilH)$ would help the algorithm to converge faster.
\end{remark}

\section{Numerical Examples}\label{sec:numerical}

\begin{figure*}[!t]
    \centering
    \includegraphics[width=0.9\linewidth]{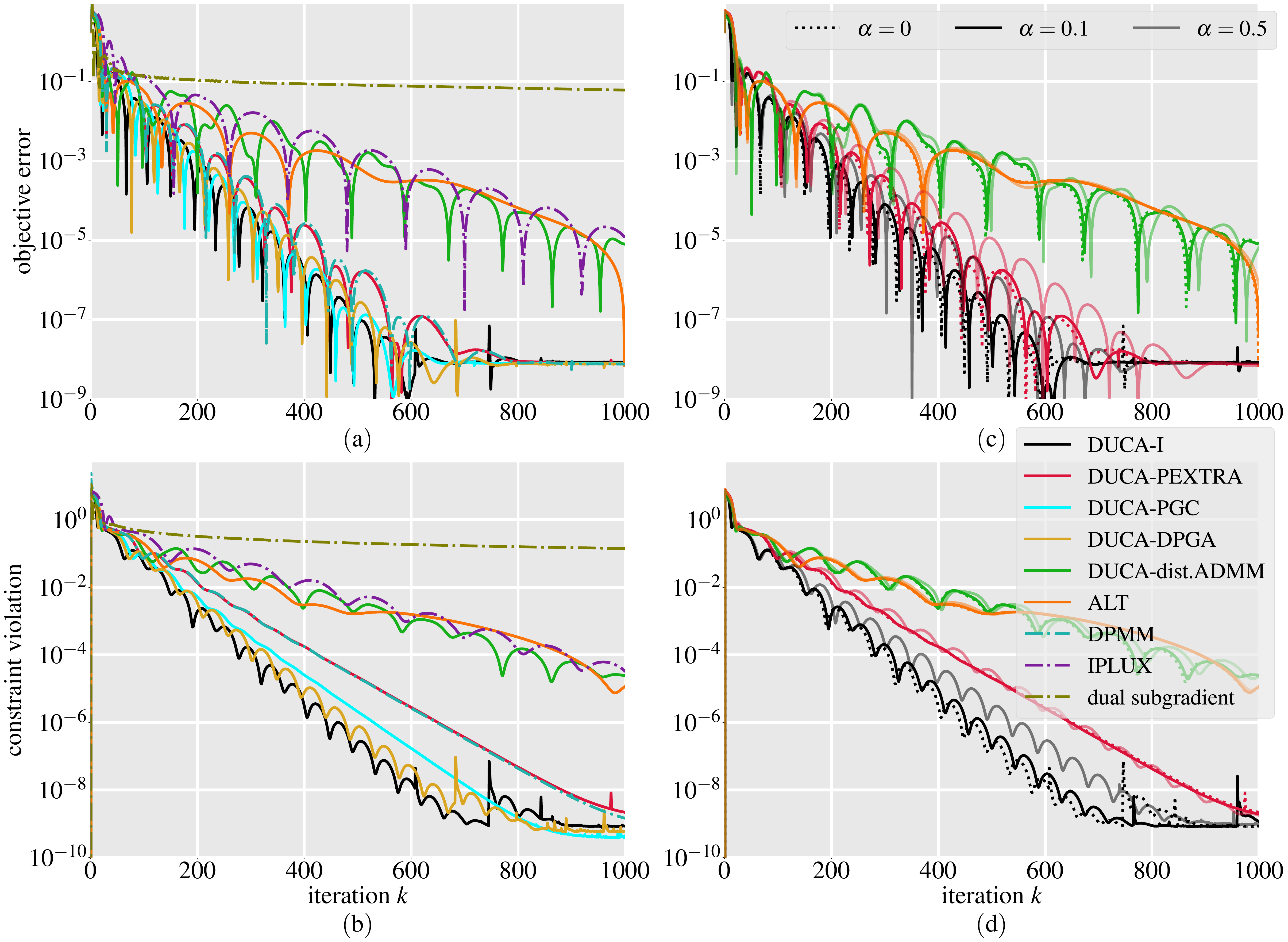}
    \caption{Convergence performance of DUCA with different parameter settings and four alternative methods (left), and performance of Pro-DUCA with different values of $\alpha$ (right).}
    \label{fig:example1}
\end{figure*}

In this section, we demonstrate the practical performance of DUCA and Pro-DUCA with different parameter settings.


We consider the following problem in the form of \eqref{prob:CC} with $N=20$, $d_i=d=3$, $m=1$ and $p=5$:
\begin{align}
    \begin{split}
        \underset{x_i\in X_i\ \forall i=1,\ldots, N}{\text{minimize}}\quad & 
        \sum_{i=1}^N x_i^T P_i x_i + Q_i^T x_i + \norm{x_i}_1 \allowdisplaybreaks\\
        \text{subject to}\quad &\sum_{i=1}^N \norm{x_i-a_i'}^2 - c'_i \leq 0,\allowdisplaybreaks \\
        &\sum_{i=1}^N B_i x_i = \0_p.
    \end{split}\label{eq:example1}
\end{align}
In the objective function, $P_i\in\R^{d\times d}$ is symmetric and semidefinite and $Q_i\in\R^d$. The local constraint sets is defined as $X_i=\{x_i\in\R^d\ |\ \norm{x-a_i}^2\leq c_i\}$, with $a_i\in\R^d$ and $c_i>\norm{a_i}^2$. In the coupled constraints, $a'_i\in\R^d$, $c'_i\in\R$ and $B_i\in\R^{p\times d}$. We also require that $\sum_{i=1}^N c_i' > \sum_{i=1}^N \norm{a_i'}^2$. It is clear that problem~\eqref{eq:example1} satisfies Assumption~\ref{asp:problem_structure} and \ref{asp:compact}. Besides, Assumption~\ref{asp:slater} is satisfied with $\tilde{x_i} = \0_d$ for all $i\in\calV$.

We execute DUCA and Pro-DUCA with different parameter settings, Augmented Lagrangian Tracking (ALT) \cite{FaP22}, DPMM \cite{GoZ23}, IPLUX \cite{WWL23} and the distributed subgradient method \cite{LWY21} to solve \eqref{eq:example1}. All parameter settings of DUCA are listed in Table~\ref{tab:param}, 
in which $M_{\mathcal{G}}$, $\mathcal{L}_1$ and $\mathcal{L}_2$ are graph-Laplacian-type matrices: $(M)_{ij}=(M)_{ji}<0$ for $\{i,j\}\in\mathcal{E}$, $(M)_{ii}=-\sum_{j\in\mathcal{N}_i} (M)_{ij}$, and $(M)_{ij}=0$ otherwise. In particular, for any $\{i,j\}\in\mathcal{E}$, $(M_{\mathcal{G}})_{ij}=-\frac{1}{\max\{|\calN_i|,|\calN_j|\}+1}$, $(\mathcal{L}_1)_{ij}=-2\rho'$ for some $\rho'>0$, and $(\mathcal{L}_2)_{ij}=-\frac{1}{2}\sqrt{\frac{cN}{|\calE|\min_k{|\calN_k|}}}$ for some $c>0$. The matrices $\Lambda_1$, $P_{D_2}$ and $P_{D_3}$ are diagonal, with  their $i$th diagonal elements defined as: $(\Lambda_1)_{ii}=(\mathcal{L}_1)_{ii}$, $(P_{D_2})_i=|\mathcal{N}_i|\sqrt{\frac{cN}{|\calE|\min_k{|\calN_k|}}}$, and  $(P_{D_3})_i=\sum_{j\in\calV}(|\calN_j|+1)(M_\calG)_{ij}^2$ respectively. Moreover, ALT is also converted into the form of DUCA, with $W_4=I-\frac{M_\calG}{2}$. All these settings satisfy their corresponding requirements (see \cite[Section~III]{WuL23}) as well as Assumption~\ref{asp:matrices} and \ref{asp:diagonal}.

In the simulation, we randomly generate a connected graph with 20 nodes and 40 links. The problem data is also randomly generated under the assumptions of \eqref{eq:example1}. All the algorithm parameters are fine-tuned within respective theoretical ranges and start from the same initial point.

Fig.~\ref{fig:example1} shows the performance of the aforementioned algorithms during 1000 iterations. The objective error is defined as $|f^\star - \sum_{i=1}^N f_i(x_i^k)|$, where $f^\star$ is the optimal value of \eqref{eq:example1} calculated by CVXPY \cite{CVXPY}. The constraint violation is the sum of $\sum_{i=1}^N\max\{\Vert x_i^k-a_i\Vert^2-c_i\}$, $\max\{\sum_{i=1}^N\Vert x_i^k-a_i'\Vert^2-c_i',0\}$ and $\Vert \sum_{i=1}^N B_ix_i^k\Vert$. The performance outcomes depicted in all sub-figures share the same underlying problem data.

\begin{table}
    \centering
    \begin{threeparttable}
        \caption{Parameter settings of DUCA in the simulation}
        \label{tab:param}
        \begin{tabular}{ccccc}
            \hline
                            &  $P_H$                &  $P_\tilH$            &  $P_D$                & $\rho$ \\
            \hline
            DUCA-I          & $M_\calG$             & $M_\calG$             & $2\rho \Lambda_\calG$ & $\rho$ \\
            DUCA-PEXTRA     & $\frac{1}{2}M_\calG$  & $\frac{1}{2}M_\calG$  & $\rho I$              & $\rho$ \\
            DUCA-PGC        & $\frac{1}{2}\calL_1$  & $\frac{1}{2}\calL_1$  & $\Lambda_1$           & $1$    \\
            DUCA-DPGA       & $\calL_2$             & $\calL_2$             & $P_{D_2}$         & $1$    \\
            DUCA-dist.ADMM  & $M_\calG^2$           & $M_\calG^2$           & $P_{D_3}$         & $\rho$ \\
            ALT             & $I-W_4^2$           & $(I-W_4)^2$           & $\rho I$              & $\rho$ \\
            \hline
        \end{tabular}
    \end{threeparttable}
\end{table}

Fig.~\ref{fig:example1}(a) and (b) highlight the superior performance of the single-exchange implementations of DUCA, i.e., DUCA-I (a novel design of parameters), DUCA-DPGA, DUCA-PGC and DUCA-PEXTRA in terms of convergence speed with respect to both optimality and feasibility, while DUCA-dist.ADMM also demonstrates a good performance that is comparable with ALT and IPLUX. Considering the communication cost per iteration, the single-exchange implementations of DUCA, DPMM, IPLUX, and the distributed dual subgradient method require each agent to send and receive $(m+p)=6$ real numbers; while DUCA-dist.ADMM and ALT require each agent to send and receive $2(m+p)=12$ real numbers. This suggests that all single-exchange implementations of DUCA also provide high communication efficiency.

Fig.~\ref{fig:example1}(c) and Fig.~\ref{fig:example1}(d) illustrate the effect of $\alpha$ on Pro-DUCA with parameter settings of DUCA-I, PGC, distributed ADMM and ALT. When $\alpha=0.1$, the performance of Pro-DUCA (solid lines) almost aligns with that of DUCA (dashed lines). Although the additional proximal term in Pro-DUCA could lead to better efficiency in the resolution of the subproblems in each iteration, setting $\alpha$ too large could compromise the convergence speed. Indeed, when $\alpha=0.5$, Pro-DUCA (solid translucent lines) performs worse, accompanied by increased oscillation. Therefore, to maximize the computational efficiency of Pro-DUCA, it is advisable to set $\alpha$ to an appropriate low level. 

\begin{remark}
    In practice, the single-exchange implementations usually outperform the double-exchange implementations. This is consistent with our convergence analysis: the constant term in the convergence rate results of the optimality error is proportional to $\frac{1}{\lambda_{N-1}(P_\tilH)}$ (see Remark~\ref{rem:weight_matrices_effect}), Compared with the single-exchange implementation that $P_\tilH=\calL$, we set $P_\tilH=\calL^2$ in the double-exchange implementation, making $\lambda_{N-1}(P_\tilH)$ much smaller, thus the constant term is larger and worse. The same analysis also applies to feasibility error and the effect of $P_A$: smaller possible value of $\lam_1(P_A)$ also makes the single-exchange implementation converge faster. 
\end{remark}





    

\section{Conclusion} \label{sec:conclusion}

We have presented the unified DUal Consensus Algorithm (DUCA) and its proximal variant, Pro-DUCA, to address distributed convex optimization with globally-coupled constraints. By leveraging a diverse range of parameter settings, DUCA and Pro-DUCA not only seamlessly adapt a variety of established consensus optimization methods to the dual of our problem, but also facilitate the development of new efficient algorithms for solving the problem. Moreover, their $O(1/k)$ convergence rates in terms of both optimality and feasibility are superior, providing new or stronger convergence results for a collection of existing methods. Simulations demonstrate the practical efficiency of the proposed methods.


\appendix

Notation used in the proofs: The affine functions in problem~\eqref{prob:CC} are represented by $h_i(x_i)=B_i x_i + c_i$, where $B_i\in\R^{p\times d_i}$ and $c_i\in\R^p\ \forall i\in\calV$. $\bsig^k$ and $\v^k$ are defined below Proposition~\ref{prop:relations}. We partition $\bsig^{k}, \z^{k}, \v^k\in\R^{N(m+p)}$ into $\bsig^{k}_\mu, \z^k_\mu, \v_\mu^k \in\R^{Nm}$ and $\bsig_\lam^{k}, \z_\lam^k, \v_\lam^k\in\R^{Np}$ as $\y_\mu$ and $\y_\lam$ (see below problem~\eqref{prob:D_prime}). Moreover, we also denote $\hatby^k:=D^{-1}(A\y^k-\tilH^\half\z^k)$ as in Proposition~\ref{prop:relations}.

Similar to \cite{Roc76}, we define two Lagrangians. The augmented Lagrangian $L:\R^{\sum d_i} \times \R^{N(m+p)} \rightarrow (-\infty, \infty]$ is defined as:
\begin{multline}
    L(\x,\y) := f(\x) + \half \Big(
        \norm{ \calP_\calK[\y+D^{-1}\tilg(\x)]}^2_D \\- \norm{\y}_D^2
    \Big) + \delta_X(\x).\label{eq:def_AL}
\end{multline} 
Here, $\delta_X(\cdot)$ is the indicator function with respect to $X$, where $\delta_X(x)=0$ if $x\in X$ and $\delta_X(x)=+\infty$ if $x\notin X$. Furthermore, the ordinary Lagrangian function is defined in an extended form:
\begin{align}
    \ell(\x,\y) := \begin{cases}
        f(\x) + \la\y, \tilg(\x)\ra, & \text{ if } \x\in X \text{ and } \y\in\calK,\\
        -\infty, &\text{ if } \x\in X \text{ and } \y\notin\calK,\\
        \infty, &\text{ if } \x\notin X.
    \end{cases}\label{eq:def_OL}
\end{align}
By direct computation, one can verify that when $\x\in X$, 
\begin{gather}
    L(\x,\y) = \max_{\y^\prime\in\R^{N(m+p)}}\left\{ \ell(\x,\y^\prime) - \half\norm{\y^\prime-\y}_D^2 \right\}, \label{eq:lagrangians_relation}
\end{gather}
where the maximum is attained uniquely at $\y^\prime=\calP_\calK[\y + D^{-1}\tilg(\x)]$. Therefore, plugging $\x^{k+1}$ and $\hatby^k$ into \eqref{eq:lagrangians_relation} yields
\begin{gather}
    L(\x^{k+1},\hatby^k) = \ell(\x^{k+1},\y^{k+1})-\half\norm{\y^{k+1}-\hatby^k}_D^2. \label{eq:lagrangians_relation_k}
\end{gather}
Moreover, when $\x\in X$ and $\y\in\calK$, we also have
\begin{align}
    \partial_\x L(\x, \y) = \partial_\x \ell(\x, \calP_\calK[\y + D^{-1}\tilg(\x)]). \label{eq:lagrangian_subgradients_relation}
\end{align}
The proof of this fact is essentially the same as the proof of Lemma~\ref{lem:equiv_update_x}.

\subsection{Proof of Lemma~\ref{lem:compute_x}} \label{app:lem:compute_x}
To prove this lemma, we first state two utilities.


\begin{lemma}\label{lem:convex_formula} 
    Let the diagonal matrix $D\in\R^{m\times m}$ be positive definite, each entry of $g(x)=[g_{1}(x),\ldots,g_{m}(x)]^T:$ $\R^n\rightarrow \R^{m}$ be convex, and $b\in\R^m$. Then following function
    \begin{gather*}
        J(x)=\norm{\left[g(x) + b\right]_+}_D^2.
    \end{gather*}  is convex.
\end{lemma}
\begin{proof}
    Note that each entry of $\left[g(x) + b\right]_+$, denoted as $p_i(x):=\left[g_i(x) + b_i\right]_+,\ 1\leq i\leq m$, is convex. Moreover, the univariate function $(\cdot)^2$ is non-decreasing on $\R_+$. Then by \cite[Theorem~3.1.9]{Nes18a}, $\left(p_i(x)\right)^2$ is convex. Thus $J(x) = \sum_i D_{ii}\left(p_i\left(x\right)\right)^2$ is also convex.
\end{proof}

\begin{lemma}\label{lem:equiv_update_x}
    Let the diagonal matrix $D\in\R^{m\times m}$ be positive definite, each entry of $g(x)=[g_{1}(x),\ldots,g_{m}(x)]^T:$ $\R^n\rightarrow \R^{m}$ be convex, and $b\in\R^m$. Then the subdifferential of the convex function $J(x)=\frac{1}{2}\norm{\big[g(x)+b]_+}_{D}^2$ at $x$ is equal to
    \begin{align*}
        \mathbf{G}(x)\Big[D\big(g(x)+b \big)\Big]_+,
    \end{align*}
    where $\mathbf{G}(x):=[\partial g_{1}(x), \ldots, \partial g_{m}(x)]\subseteq \R^{n\times m}$.
\end{lemma}
\begin{proof}
    Notice that $k(y)=\half\norm{\max\{y, \0_{m}\}}^2_D: \R^{m}\rightarrow\R$ is differentiable: $\nabla k(y) = D\max\{y, \0\}$, and monotone: if $y_1\geq y_2$ in the component-wise sense, then $k(y_1)\geq k(y_2)$. Let $p_i(x) := g_i(x) + b_i$, i.e., the $i$th entry of $g(x) + b$. Then $J(x) = k(p_1(x),\ldots,p_{m}(x))$. By \cite[Lemma~3.1.17]{Nes18a}, we have
    \begin{align*}
        \partial J(x) &= \sum_{i}^{m}\nabla_i k(p_1(x),\ldots,p_{m}(x))\partial p_i(x) \\
        &= \sum_{i}^{m} D_{ii}\max\{p_i(x), 0\}\partial p_i(x) \\
        &= [\partial g_1(x)\cdots \partial g_{m}(x)]D \left[g(x) + b\right]_+ \\
        &= \mathbf{G}(x)\Big[D\big(g(x)+b \big)\Big]_+,
    \end{align*}
\end{proof}


Now we prove Lemma~\ref{lem:compute_x}. $-\bsig^{k+1}\in N_\calK(\y^{k+1})$ implies that
\begin{align*}
    (\bsig_\mu^{k+1})_j &\begin{cases}
        = 0, &\text{ if } (\y^{k+1}_\mu)_j\geq 0 \\
        \geq 0, &\text{ if } (\y^{k+1}_\mu)_j = 0
    \end{cases} \text{ for all } 1\leq j\leq Nm, \\
    \bsig_\lam^{k+1} &\ \ = \0_{Np}.
\end{align*}
With this observation, by considering each entry in equality \eqref{eq:y_update_derivation1}, 
we obtains 
\begin{align}
    D\y^{k+1} &= \calP_\calK\Big[A\y^k-\tilH^\half\z^k+\tilg(\x^{k+1}) \Big], \label{eq:y_update_derivation2}\\
    -\bsig^{k+1} &= \calP_{\calK^\circ}\Big[A\y^k - \tilH^\half\z^k +\tilg(\x^{k+1})\Big], \label{eq:sig_update_derivation1}
\end{align}
where $\calK^\circ:=\{-\bsig\in\R^{N(m+p)}\ |\ -\bsig_\mu\in\R_-^{Nm}, -\bsig_\lam = \0_{Np}\}$ is the polar cone of $\calK$. 

Therefore, it is left to find an $\x^{k+1} \in X(\y^{k+1})$, such that \eqref{eq:y_update_derivation2} holds. From the optimal condition of $X(\y^{k+1}) = \arg\min_{\x\in X} f(\x) + \langle \y^{k+1}, \tilg(\x) \rangle$, we have
\begin{align}
    \0 \in \partial f(\x^{k+1}) + \tilG(\x^{k+1})\y^{k+1} + N_X(\x^{k+1}), 
    \label{eq:x_update_derivation1}
\end{align}
where $\tilG(\x^{k+1})=\diag(\tilG_1(x_1^{k+1}),$$\ldots,$$\tilG_N(x_N^{k+1}))$ and $\tilG_i(x_i^{k+1})$$=$$[\partial g_{i1}(x_i^{k+1}),$$\ldots,$$\partial g_{im}(x_i^{k+1}),$$B_i^T] \subseteq$$\R^{d_i\times(m+p)}$. 

For convenience, we separate subdifferentials of inequality constraints and gradients of equality constraints. Denote $\mathbf{G}(\x^{k+1})=[\partial g_{11},$$\ldots,$$\partial g_{1m},$$\ldots,$$\partial g_{N1},$$\ldots,$$\partial g_{Nm}](\x^{k+1})\subseteq \R^{(\sum d_i)\times Nm}$ (abusing the notation of $g_{ij}$: we extend its domain from $\R^{d_i}$ to $\R^{\sum d_i}$) and $\mathbf{H}=\diag(B_1^T, \ldots, B_N^T)\subseteq \R^{(\sum d_i)\times Np}$. Moreover, due to Assumption~\ref{asp:matrices}.4), the parameter matrices work on two groups of dual variables separately, so we also denote $A_\mu=P_A\otimes I_m$, $A_\lam=P_A\otimes I_p$ and $D_\mu, D_\lam, \tilH_\mu, \tilH_\lam$ likely. With these notations, substituting \eqref{eq:y_update_derivation2} into \eqref{eq:x_update_derivation1} yields
\begin{multline}
        \0 \in \partial f(\x^{k+1}) + \mathbf{H}D_\lam^{-1}\big(A_\lam\y_\lam^k-\tilH_\lam^\half\z_\lam^k+h(\x^{k+1}) \big)\\
        + \mathbf{G}(\x^{k+1})\Big[D_\mu^{-1}\big(A_\mu\y_\mu^k-\tilH_\mu^\half\z_\mu^k+g(\x^{k+1}) \big)\Big]_+ \\+ N_X(\x^{k+1}).
\end{multline}
By Lemma~\ref{lem:equiv_update_x} and optimality condition, this is further equivalent to 
\begin{multline}
        \x^{k+1} \in \arg\min_{\x\in X} \Big\{f(\x) + \frac{1}{2}\norm{
            \big[A_\mu\y_\mu^k-\v_\mu^k+g(\x)\big]_+
        }_{D_\mu^{-1}}^2 \allowdisplaybreaks\\
        \qquad\qquad\qquad\quad +\frac{1}{2}\norm{
            A_\lam\y_\lam^k-\v_\lam^k+h(\x)
        }_{D_\lam^{-1}}^2\Big\} \allowdisplaybreaks\\
        =\arg\min_{\x\in X} \Big\{f(\x) + \frac{1}{2}\norm{
            \calP_\calK\big[A\y^k-\tilH^\half\z^k
             +\tilg(\x)\big]
        }_{D^{-1}}^2 \Big\}. \label{eq:x_update_derivation2}
\end{multline}
The diagonal assumption is necessary for this equivalence, since it ensures the convexity (Lemma~\ref{lem:convex_formula}) and differentiability (Lemma~\ref{lem:equiv_update_x}) of $\norm{[\cdot]_+}_{D_\mu^{-1}}$.
Thus any $\x^{k+1}$ computed by \eqref{eq:x_update_derivation2} satisfies that $\x^{k+1}\in X(\y^{k+1})$, where $\y^{k+1}$ is given by formula \eqref{eq:y_update_derivation2}. Such an $\x^{k+1}$ does exist, since $X$ in \eqref{eq:x_update_derivation2} is compact by Assumption~\ref{asp:compact}. Therefore, $\x^{k+1}$, $\y^{k+1}$ and $\bsig^{k+1}$ computed by \eqref{eq:x_update_derivation2}, \eqref{eq:y_update_derivation2} and \eqref{eq:sig_update_derivation1} satisfy our requirements.

\subsection{Proof of Proposition~\ref{prop:z_star}} \label{app:prop:z_star}
$(\mu^\star, \lam^\star)$ is an optimal solution of problem~\eqref{prob:D} by definition. Consequently, $\y^\star=\1_N \otimes [(\mu^\star)^T,(\lam^\star)^T]^T$ is an optimal solution of problem~\eqref{prob:D_prime}. This is because a feasible solution of \eqref{prob:D_prime} is a consensus one, and any better solution of problem~\eqref{prob:D_prime} would provide a better solution to \eqref{prob:D}.

It remains to show that $\z^\star=(\tilH^\half)^\dag\tilg(\x^\star)$ 
is a geometric multiplier (so that $(\y^\star, \z^\star)$ is an optimal solution-geometric multiplier pair). By \cite[Proposition 6.1.5]{Ber16}, $\z^\star$ is a geometric multiplier if and only if 
\begin{align}
    \y^\star &\in \arg\min_{\y\in\calK}-q(\y) + \la\z^\star, \tilH^\half\y\ra, 
\end{align}
or equivalently, 
\begin{align}
    \0 \in \partial -q(\y^\star) + \tilH^\half(\tilH^\half)^\dag\tilg(\x^\star) + N_{\calK}(\y^\star). \label{eq:lem:z_star1} 
\end{align}
Moreover, by \eqref{eq:optimal_primal_dual_pair1}, $\x^\star\in\arg\min_{\x\in X}f(\x)+\la\mu^\star, (\1_N\otimes I_m)^Tg(\x)\ra + \la\lam^\star, (\1_N\otimes I_p)^Th(\x)\ra = \arg\min_{\x\in X}f(\x)+\la\y^\star,\tilg(\x)\ra$. Then $\tilg(\x^\star)\in\partial q(\y^\star)$. Now to prove \eqref{eq:lem:z_star1}, we only need to show 
\begin{gather}
    \0 \in \tilH^\half(\tilH^\half)^\dag\tilg(\x^\star) - \tilg(\x^\star) + N_\calK(\y^\star).\label{eq:lem:z_star2}
\end{gather}

Note that $\tilH^\half(\tilH^\half)^\dag$ is the orthogonal projection onto $\calR(\tilH^\half)=\calR(\tilH)=(\calN(\tilH))^\perp = S^\perp$, where $S^\perp$ is the orthogonal complement of $S$ in \eqref{eq:nullspace_S}: 
\begin{gather*}
    S^\perp = \{(y_1,\ldots,y_N)\in\R^{N(m+p)}\ |\ \sum_{i=1}^N y_i=\0_{m+p}\}.
\end{gather*}
Thus \begin{multline}
    \tilH^\half\z^\star=
    \tilH^\half(\tilH^\half)^\dag\tilg(\x^\star) = \tilg(\x^\star)-\1_N\otimes(\frac{1}{N}\sum_{i=1}^N \tilg_i(x_i^\star)).  \label{eq:lem:z_star_v_star_def}
\end{multline}
Now to prove \eqref{eq:lem:z_star2}, we only need to show $\1_N\otimes(\frac{1}{N}\sum_{i=1}^N \tilg_i(x_i^\star))\in N_\calK(\y^\star)$, i.e.,
\begin{align}
    \begin{bmatrix}
        \frac{1}{N}\sum_{i=1}^N g_i(x_i^\star) \\
        \frac{1}{N}\sum_{i=1}^N h_i(x_i^\star) \\
        \vdots \\
        \frac{1}{N}\sum_{i=1}^N g_i(x_i^\star) \\
        \frac{1}{N}\sum_{i=1}^N h_i(x_i^\star)
    \end{bmatrix} \in \begin{bmatrix}
        N_{\R_+^{m}}(\mu^\star) \\
        N_{\R^p}(\lam^\star) \\
        \vdots \\
        N_{\R_+^{m}}(\mu^\star) \\
        N_{\R^p}(\lam^\star)
    \end{bmatrix}.
\end{align}
This is done by noticing $\frac{1}{N}\sum_{i=1}^N h_i(x_i^\star)=\0_p\in\{\0_p\}=N_{\R^p}(\lam^\star)$ and the complementary slackness of $\sum_{i=1}^N g_i(x_i^\star)$ and $\mu^\star$.

Therefore, $(\y^\star, \z^\star)$ is an optimal solution-geometric multiplier pair. It follows that $\z^\star$ is dual optimal and there is no duality gap (see the proof in \cite[Proposition 6.1.5]{Ber16}).

\subsection{Proof of Proposition~\ref{prop:relations}} \label{app:prop:relations}
\begin{enumerate}
    \item The update formula of $\y^k$ \eqref{eq:alg1_y_update} implies $D\y^{k+1} = \calP_\calK\big[A\y^k-\tilH^\half\z^k+\tilg(\x^{k+1})\big]$, since $D\succ 0$ is diagonal; the definition of $\bsig^k$ \eqref{eq:sig_update} implies $-\bsig^{k+1} = \calP_{\calK^\circ}\big[A\y^k-\tilH^\half\z^k+\tilg(\x^{k+1})\big]$. Then we have \begin{align*}
        D\y^{k+1}-\bsig^{k+1} = A\y^k-\tilH^\half\z^k+\tilg(\x^{k+1}).
    \end{align*} This can be directly verified by looking at each entry in $\y^{k+1}$ and $\bsig^{k+1}$, since $\calK$ and $\calK^\circ$ have very simple structures: $\calK$ is a direct product of $\R$'s and $\R_+$'s, and $\calK^\circ$ is a direct product of $\{0\}$'s and $\R_-$'s.  
    \item \eqref{eq:relation_iterates} is equivalent to $D\y^{l+1}-A\y^l = \tilg(\x^{l+1})+\bsig^{l+1}-\tilH^\half\z^l$. Multiplying $(\1_N\otimes I_{m+p})^T$ to both sides of it yields 
    \begin{multline*}
        (\1_N\otimes I_{m+p})^T A(\y^{l+1}-\y^l) \\= (\1_N\otimes I_{m+p})^T\big(\tilg(\x^{l+1})+\bsig^{l+1}\big),
    \end{multline*} 
    since $D=A+\rho H$, and by Assumption~\ref{asp:matrices}.2), $(\1_N\otimes I_{m+p})^TH = (\1_N\otimes I_{m+p})^T\tilH = \0_{(m+p)\times N(m+p)}$. 
    

    \eqref{eq:relation_y_A_constraints} is obtained by summing the relation above over $l=0,1,\ldots,k-1$.

    \item We analyze equality constraints and inequality constraints separately. Notice that $\bsig_\lam^k=\0_p$ for all $k\geq 1$, due to the definition of $\bsig^k$ \eqref{eq:sig_update} and the structure of $\calK^\circ$. For equality constraints, we have
    \begin{align}
        &(\1_N\otimes I_p)^T h(\barbx^k) \overset{\text{(a)}}{=} (\1_N\otimes I_p)^T \frac{1}{k}\sum_{l=1}^k h(\x^{l}) \nonumber \allowdisplaybreaks\\
        \overset{\text{(b)}}{=}{} &(\1_N\otimes I_p)^T \frac{1}{k}\sum_{l=1}^k \left(h(\x^{l}) + \bsig_\lam^{l}\right) \nonumber \allowdisplaybreaks\\
        \overset{\text{(c)}}{=}{} &\frac{1}{k}(\1_N\otimes I_p)^T (P_A\otimes I_p) (\y_\lam^{k} - \y_\lam^0), \nonumber 
    \end{align}
    where (a) is because $h_i$'s are affine; (b) is due to $\bsig_\lam^l=\0_p$; (c) is by \eqref{eq:relation_y_A_constraints}. Thus $\norm{(\1_N\otimes I_p)^T h(\barbx^k)}=\frac{1}{k}\norm{(\1_N\otimes I_p)^T (P_A\otimes I_p) (\y_\lam^{k} - \y_\lam^0)}$.
    
    For inequality constraints, we have following entrywise relation:
    \begin{align}
        &(\1_N\otimes I_m)^T g(\barbx^k) \overset{\text{(a)}}{\leq} (\1_N\otimes I_m)^T \frac{1}{k}\sum_{l=1}^k g(\x^{l}) \nonumber\\
        \overset{\text{(b)}}{\leq}{} &(\1_N\otimes I_m)^T \frac{1}{k}\sum_{l=1}^k \left(g(\x^{l}) + \bsig_\mu^{l}\right) \nonumber\\
        \overset{\text{(c)}}{=}{} &\frac{1}{k}(\1_N\otimes I_m)^T (P_A\otimes I_m) (\y_\mu^{k} - \y_\mu^0) \nonumber 
    \end{align}
    where (a) is by convexity of $g_i$'s; (b) is due to nonnegativity of $\bsig^l$'s; (c) is by \eqref{eq:relation_y_A_constraints}. Thus $\big[(\1_N\otimes I_m)^T g(\barbx^k)\big]_+ \leq \frac{1}{k}\big[(\1_N\otimes I_m)^T (P_A\otimes I_m) (\y_\mu^{k} - \y_\mu^0)\big]_+$. Since both sides are nonnegative, taking the norm of them yields $\norm{\big[(\1_N\otimes I_m)^T g(\barbx^k)\big]_+} \leq \frac{1}{k}\norm{\big[(\1_N\otimes I_m)^T (P_A\otimes I_m) (\y_\mu^{k} - \y_\mu^0)\big]_+} \leq \frac{1}{k}\norm{(\1_N\otimes I_m)^T (P_A\otimes I_m) (\y_\mu^{k} - \y_\mu^0)}$. Combing this with the result of equality constraints yields $
        \left\Vert\begin{bmatrix}
            \big[\sum_{i=1}^N g_i(\barx_i^k)\big]_+ \\
            \sum_{i=1}^N h_i(\barx_i^k)
        \end{bmatrix}\right\Vert
        \leq{}  \frac{1}{k} \lVert(\1_N\otimes I_{m+p})^TA(\y^k-\y^0)\rVert 
        \leq{}  \frac{1}{k} \norm{(\1_N\otimes I_{m+p})} \norm{A^\half} \norm{\y^k-\y^0}_A 
        ={} \frac{\sqrt{N\lam_{1}(P_A)}}{k} \norm{\y^k-\y^0}_A.
    $

    \item By Assumption~\ref{asp:slater}, strong duality holds, and we have an optimal primal-dual solution pair $(\x^\star, y^\star)$, where $y^\star=[(\mu^\star)^T,(\lam^\star)^T]^T\in\R^{m+p}$ satisfies \eqref{eq:optimal_primal_dual_pair1}, thus
    \begin{align}
        &f(\x^\star) - f(\barbx^k) \leq \la y^\star, (\1_{N}\otimes I_{m+p})^T \tilg(\barbx^k) \ra \nonumber \allowdisplaybreaks\\
        ={} &\la\mu^\star, (\1_N\otimes I_{m})^T g(\barbx^k)\ra + \la\lam^\star, (\1_N\otimes I_{p})^T h(\barbx^k)\ra \nonumber \allowdisplaybreaks\\
        \overset{\text{(a)}}{\leq}{} &\la\mu^\star, \big[(\1_N\otimes I_{m})^T g(\barbx^k)\big]_+\ra \nonumber \allowdisplaybreaks\\
        &\qquad\qquad\qquad\qquad\quad\ \  +\la\lam^\star, (\1_N\otimes I_{p})^T h(\barbx^k)\ra \nonumber \allowdisplaybreaks\\
        ={} & \left\la y^\star, \begin{bmatrix}
            \big[(\1_N\otimes I_{m}) g(\barbx^k)\big]_+\\
            (\1_N\otimes I_{p}) h(\barbx^k)
        \end{bmatrix} \right\ra \nonumber \allowdisplaybreaks\\
        ={} & \left\Vert y^\star \right\Vert\cdot \left\Vert\begin{bmatrix}
            \big[(\1_N\otimes I_{m}) g(\barbx^k)\big]_+\\
            (\1_N\otimes I_{p}) h(\barbx^k)
        \end{bmatrix} \right\Vert, \nonumber
    \end{align}
    where (a) is due to $\mu^\star\geq \0_m$, ; (b) is by CBS inequality.

    \item We have 
    \begin{align*}
        &\la\y^{k+1}, D(\hatby^k - \y^{k+1}) +\v^\star\ra \allowdisplaybreaks\\ 
        \overset{\text{(a)}}{=}& \la\y^{k+1}, A\y^k - D\y^{k+1} -\v^k+\v^\star\ra \allowdisplaybreaks\\
        \overset{\text{(b)}}{=}& \la\y^{k+1}, A(\y^k-\y^{k+1})\ra 
        + \la\y^{k+1}, \v^\star-\v^{k+1}\ra \allowdisplaybreaks\\
        &\qquad +\la\y^{k+1}, -\rho{H}\y^{k+1} +\v^{k+1}-\v^k\ra \allowdisplaybreaks\\
        \overset{\text{(c)}}{=}& \la\y^{k+1}, A(\y^k-\y^{k+1})\ra 
        + \la\y^{k+1}, \v^\star-\v^{k+1}\ra \allowdisplaybreaks\\
        &\qquad +\la\y^{k+1}, \rho(\tilH-H)\y^{k+1}\ra,
    \end{align*}
    where (a) is by the definition of $\hatby^k$; (b) is by $D=A+\rho H$; (c) is due to \eqref{eq:v_update}. Notice that the third term is less than or equal to zero, since $\tilH\preceq H$, by Assumption~\ref{asp:matrices}.3). Then
    \begin{align*}
        &\la\y^{k+1}, D(\hatby^k - \y^{k+1}) +\v^\star\ra \\ 
        \overset{\text{}}{\leq}& \la\y^{k+1}, A(\y^k-\y^{k+1})\ra 
        + \la\y^{k+1}, \v^\star-\v^{k+1}\ra \\
        \overset{\text{(a)}}{=}& \la\y^{k+1}, A(\y^k-\y^{k+1})\ra \\
        &\qquad + \frac{1}{\rho}\la\tilH^\dag(\v^{k+1}-\v^k), \v^\star-\v^{k+1}\ra \\
        \overset{\text{(b)}}{\leq}& \frac{1}{2}(\norm{\y^k}_A^2 - \norm{\y^{k+1}}_A^2) \\
        &\qquad + \frac{1}{2\rho}(\norm{\v^k-\v^\star}_{\tilH^\dag}^2 - \norm{\v^{k+1}-\v^\star}_{\tilH^\dag}^2),
    \end{align*}
    where (a) and (b) are both due to Lemma~1 from \cite{WWL23}.
\end{enumerate}

\subsection{Proof of Lemma~\ref{lem:alg1_objective_error1}} \label{app:lem:alg1_objective_error1}
We denote $\hatby^k:=D^{-1}(A\y^k-\v^k)$. Consider
\begin{align} 
    &f(\x^{k+1}) -\half \left( \norm{\hatby^k}_D^2 - \norm{\y^{k+1}}_D^2 \right) \nonumber \allowdisplaybreaks\\
    \overset{\text{(a)}}{=}&f(\x^{k+1}) -\half \left( \norm{\hatby^k}_D^2 - \norm{\calP_\calK[\hatby^{k} + D^{-1}\tilg(\x^{k+1})]}_D^2 \right) \nonumber \allowdisplaybreaks\\
    \overset{\text{(b)}}{=}& L(\x^{k+1},\hatby^k) \nonumber \allowdisplaybreaks\\
    \overset{\text{(c)}}{=}& \ell(\x^{k+1},\y^{k+1}) - \half\norm{\y^{k+1}-\hatby^k}_D^2 \nonumber \allowdisplaybreaks\\
    \overset{\text{(d)}}{=}& q(\y^{k+1}) - \half\norm{\y^{k+1}-\hatby^k}_D^2, \label{eq:lem:oe_part1.1}
\end{align}
where (a) is due to the update formula of $\y^{k+1}$ \eqref{eq:alg1_y_update}; (b) is by the definition of the augmented Lagrangian \eqref{eq:def_AL}; (c) is due to \eqref{eq:lagrangians_relation_k}; (d) is because $q(\y^{k+1})=\min_{\x\in X}\ell(\x,\y^{k+1})$ and $\x^{k+1}\in X(\y^{k+1})$.

One the other hand, consider the optimal primal-dual pair $(\y^\star, \z^\star)$ provided by Proposition~\ref{prop:z_star}, which gives: $ -q(\y^{k+1}) 
    + \la \z^\star, \tilH^\half \y^{k+1} \ra 
    \geq -q(\y^\star) 
+ \la \z^\star, \tilH^\half \y^\star \ra$. Combining this with $\tilH^\half\y^\star=\0$ and $q(\y^\star) = f(\x^\star)$ (strong duality is provided by Assumption~\ref{asp:slater}), we have 
$
    q(\y^{k+1}) \leq f(\x^\star) + \la \v^\star, \y^{k+1} \ra.
$ 
Together with \eqref{eq:lem:oe_part1.1}, this yields
\begin{align}
    &f(\x^{k+1}) - f(\x^\star) \nonumber\\
    \leq& \half (\norm{\hatby^k}_D^2 - \norm{\y^{k+1}}_D^2 - \norm{\hatby^k-\y^{k+1}}_D^2) + \la\v^\star, \y^{k+1} \ra \nonumber\\
    =& \la \y^{k+1}, D(\hatby^k-\y^{k+1}) + \v^\star\ra. \label{eq:lem:oe_part1.2}
\end{align}

The proof is done by combing the relation above with \eqref{eq:relation_v_star_lyapunov}.

\subsection{Proof of Lemma~\ref{lem:alg2_objective_error}} \label{app:lem:alg2_objective_error}
\noindent\textbf{Step 1:} we show that $\ell(\x^{k+1},\y^{k+1})\leq \ell(\x^\star,\y^\star)+\la \y^{k+1}, \v^\star\ra - \la \alpha(\x^k-\x^{k+1}), \x^\star-\x^{k+1}\ra$.

Note that the update formula of $\x^{k+1}$ \eqref{eq:alg2_x_update} can be written as $\x^{k+1} = \arg\min_{\x\in \R^{\sum d_i}} L(\x, \hatby^{k}) + \frac{\alpha}{2}\norm{\x-\x^k}^2$, implying $\alpha(\x^k-\x^{k+1}) \in \partial_\x L(\x^{k+1}, \hatby^k)$. Combing this with \eqref{eq:lagrangian_subgradients_relation} gives $\alpha(\x^k-\x^{k+1}) \in \partial_\x \ell(\x^{k+1}, \y^{k+1})$. Considering that $\ell(\x,\y)$ is convex in $\x$, we further have $\ell(\x^\star, \y^{k+1}) \geq \ell(\x^{k+1}, \y^{k+1}) + \la \alpha(\x^k-\x^{k+1}), \x^\star-\x^{k+1}\ra$. 
Thus
\begin{align}
    &\ell(\x^{k+1}, \y^{k+1}) - \ell(\x^\star,\y^\star) \nonumber\\
    &\qquad\qquad\qquad\qquad + \la \alpha(\x^k-\x^{k+1}), \x^\star-\x^{k+1}\ra  \nonumber\\
    \leq{} &\ell(\x^\star,\y^{k+1}) - \ell(\x^\star,\y^\star) \nonumber\\
    ={} &\la \y^{k+1}, \tilg(\x^\star) \ra - \la \y^\star, \tilg(\x^\star)\ra \nonumber\\
    ={} &\la \y^{k+1}, \tilg(\x^\star) \ra - \la\mu^\star, (\1_N \otimes I_m)^T g(\x^\star)\ra \nonumber\\
    &\qquad\qquad\qquad\qquad\qquad- \la\lam^\star, (\1_N \otimes I_p)^T h(\x^\star)\ra \nonumber\\
    \overset{\text{(a)}}{=}{} &\la \y^{k+1}, \tilg(\x^\star) \ra \nonumber\\
    \overset{\text{(b)}}{=}{} &\la \y^{k+1}, \v^\star\ra + \la \y^{k+1}_\mu, \1_N\otimes(\frac{1}{N}\sum_{i=1}^N g_i(x_i^\star))\ra \nonumber\\
    &\qquad\qquad\qquad+ \la \y^{k+1}_\lam, \1_N\otimes(\frac{1}{N}\sum_{i=1}^N h_i(x_i^\star))\ra \nonumber \\
    \overset{\text{(c)}}{\leq}{} &\la \y^{k+1}, \v^\star\ra, \label{eq:lem:alg2_oe_part1}
\end{align}
where (a) is because the complementary slackness of $\mu^\star$ and $(\1_N \otimes I_m)^T g(\x^\star)$, and $(\1_N \otimes I_p)^T h(\x^\star)=\0_p$; (b) is due to the formula of $\v^\star=\tilH^\half\z^\star$ in \eqref{eq:lem:z_star_v_star_def}; (c) is by $\y_\mu^{k+1}\geq 0$, $\sum_{i=1}^N g_i(x_i^\star)\leq \0$ and $\sum_{i=1}^N h_i(x_i^\star)=\0$.

\bigskip

\noindent\textbf{Step 2.}
By \eqref{eq:lagrangians_relation_k}, $f(\x^{k+1})= \ell(\x^{k+1},\y^{k+1}) + \half(\norm{\hatby^k}_D^2 - \norm{\y^{k+1}}_D^2 - \norm{\y^{k+1}-\hatby^k}_D^2) = \ell(\x^{k+1},\y^{k+1}) + \la\y^{k+1}, D(\hatby^k-\y^{k+1})\ra$. Together with \eqref{eq:lem:alg2_oe_part1}, this yields
\begin{align*}
    &f(\x^{k+1}) - \ell(\x^\star,\y^\star) \allowdisplaybreaks\\
    \leq{} & \la\y^{k+1}, D(\hatby^k-\y^{k+1})+\v^\star\ra \\
        &\qquad\qquad\qquad\qquad+ \la \alpha(\x^k-\x^{k+1}), \x^{k+1}-\x^\star\ra \\
    \overset{\text{(a)}}{\leq}{} &  \frac{1}{2\rho}(\norm{\v^k-\v^\star}_{\tilH^\dag}^2 - \norm{\v^{k+1}-\v^\star}_{\tilH^\dag}^2)  + \frac{1}{2}(\norm{\y^k}_A^2  \allowdisplaybreaks\\
        &\qquad- \norm{\y^{k+1}}_A^2) + \la \alpha(\x^k-\x^{k+1}), \x^{k+1}-\x^\star\ra \\
    \overset{\text{(b)}}{\leq}{} &  \frac{1}{2\rho}(\norm{\v^k-\v^\star}_{\tilH^\dag}^2 - \norm{\v^{k+1}-\v^\star}_{\tilH^\dag}^2)  + \frac{1}{2}(\norm{\y^k}_A^2  \allowdisplaybreaks\\
        &\qquad- \norm{\y^{k+1}}_A^2)+ \frac{\alpha}{2}(\norm{\x^k-\x^\star}^2 - \norm{\x^{k+1}-\x^\star}^2),
\end{align*}
where (a) is due to \eqref{eq:relation_v_star_lyapunov}; (b) is because $\la\alpha(\x^k-\x^{k+1}), \x^{k+1}-\x^\star\ra = \frac{\alpha}{2}(\norm{\x^k-\x^\star}^2 - \norm{\x^{k+1}-\x^\star}^2- \norm{\x^k-\x^{k+1}}^2 )\leq \frac{\alpha}{2}(\norm{\x^k-\x^\star}^2 - \norm{\x^{k+1}-\x^\star}^2)$.

The proof is completed by noticing that $\ell(\x^\star,\y^\star) = f(\x^\star) + \la\y^\star, \tilg(\x^\star)\ra = f(\x^\star) + \la\mu^\star, (\1_N\otimes I_{m+p})^T\tilg(\x^\star)\ra = f(\x^\star)$.

\subsection{Proof of Lemma~\ref{lem:alg2_y_bounded}} \label{app:lem:alg2_y_bounded}
By Assumption~\ref{asp:slater}, strong duality holds and we have an optimal primal-dual solution pair $(\x^\star, y^\star)$ satisfying \eqref{eq:optimal_primal_dual_pair1}, thus $f(\x^\star)\leq f(\x^l) + \la y^\star, (\1_N\otimes I_{m+p})^T \tilg(\x^l)$, for all $l\geq 1$. Summing this relation over $l=1,2,\ldots,k$ yields
\begin{align}
    \sum_{l=1}^k \big(f(\x^\star)-f(\x^l)\big) \leq \la y^\star, (\1_N\otimes I_{m+p})^T \sum_{l=1}^k\tilg(\x^l) \ra. \label{eq:lem:alg2_y_bounded1}
\end{align}
Notice that $\la y^\star, (\1_N\otimes I_{m+p})^T \sum_{l=1}^k\bsig^l\ra = \la\mu^\star, (\1_N\otimes I_{m})^T \sum_{l=1}^k\bsig_\mu^l\ra \geq 0$, since $\bsig_\mu^l\geq 0$, $\bsig_\lam^l=\0$ and $\mu^\star\geq \0$. Together with \eqref{eq:lem:alg2_y_bounded1}, this yields
\begin{align}
    &\sum_{l=1}^k \big(f(\x^\star)-f(\x^l)\big) \nonumber\\
    \leq{} &\la y^\star, (\1_N\otimes I_{m+p})^T \sum_{l=1}^k\big(\tilg(\x^l)+\bsig^l\big) \ra \nonumber\\
    \overset{\text{(a)}}{=}{} & \la y^\star, (\1_N\otimes I_{m+p})^T A(\y^k-\y^0) \ra \nonumber\\
    \overset{\text{}}{=}{} & \la \1_N\otimes y^\star, A(\y^k-\y^0) \ra \nonumber\\
    \overset{\text{(b)}}{\leq}{} & \norm{\1_N\otimes y^\star} \norm{A^\half} \norm{A^\half(\y^k-\y^0)} \\
    \overset{\text{}}{=}{} & \sqrt{N\lam_{1}(P_A)}\norm{y^\star} \norm{\y^k-\y^0}_A \nonumber\\
    \overset{\text{(c)}}{\leq}{} & \sqrt{N\lam_{1}(P_A)}\norm{y^\star} (\norm{\y^k}_A+\norm{\y^0}_A) \label{eq:lem:alg2_y_bounded2}
\end{align}
where (a) is by \eqref{eq:relation_y_A_constraints}; (b) is by CBS inequality; (c) is by triangular inequality. 

One the other hand, Lemma~\ref{lem:alg2_objective_error} implies that $\sum_{l=1}^k\big(f(\x^l)-f(\x^\star)\big) \leq R^0 -R^k$, for all $k\geq 1$, where we denote $R^l:=\frac{1}{2\rho}\norm{\v^l-\v^\star}_{\tilH^\dag}^2 + \half\norm{\y^l}_A^2 + \frac{\alpha}{2}\norm{\x^l-\x^\star}^2$ for all $l\geq 0$. 
This together with \eqref{eq:lem:alg2_y_bounded2} gives $R^k-R^0 \leq C_1(\norm{\y^k}_A + \norm{\y^0}_A)$, where $C_1:=\sqrt{N\lam_{1}(P_A)}\norm{y^\star}$. Therefore, for any $k\geq 1$, we have
\begin{multline*}
    \half(\norm{\y^k}_A-C_1)^2 + \frac{\alpha}{2}\norm{\x^k-\x^\star}^2 + \frac{1}{2\rho}\norm{\v^k-\v^\star}_{\tilH^\dag}^2 \\ \leq \half(\norm{\y^0}_A+C_1)^2 + \frac{\alpha}{2}\norm{\x^0-\x^\star}^2 + \frac{1}{2\rho}\norm{\v^0-\v^\star}_{\tilH^\dag}^2,
\end{multline*} 
which directly derives \eqref{eq:lem:alg2_y_bounded_result}.

\bibliographystyle{IEEEtran}
\bibliography{mybib}

\begin{thebibliography}{10}
\providecommand{\url}[1]{#1}
\csname url@samestyle\endcsname
\providecommand{\newblock}{\relax}
\providecommand{\bibinfo}[2]{#2}
\providecommand{\BIBentrySTDinterwordspacing}{\spaceskip=0pt\relax}
\providecommand{\BIBentryALTinterwordstretchfactor}{4}
\providecommand{\BIBentryALTinterwordspacing}{\spaceskip=\fontdimen2\font plus
\BIBentryALTinterwordstretchfactor\fontdimen3\font minus \fontdimen4\font\relax}
\providecommand{\BIBforeignlanguage}[2]{{%
\expandafter\ifx\csname l@#1\endcsname\relax
\typeout{** WARNING: IEEEtran.bst: No hyphenation pattern has been}%
\typeout{** loaded for the language `#1'. Using the pattern for}%
\typeout{** the default language instead.}%
\else
\language=\csname l@#1\endcsname
\fi
#2}}
\providecommand{\BIBdecl}{\relax}
\BIBdecl

\bibitem{Ned20}
A.~Nedić, ``Distributed gradient methods for convex machine learning problems in networks: Distributed optimization,'' \emph{IEEE Signal Processing Magazine}, vol.~37, no.~3, pp. 92--101, 2020.

\bibitem{NAL18}
A.~Nedić and J.~Liu, ``Distributed optimization for control,'' \emph{Annual Review of Control, Robotics, and Autonomous Systems}, vol.~1, pp. 77--103, 2018.

\bibitem{NeS18}
Y.~Nesterov and V.~Shikhman, ``Dual subgradient method with averaging for optimal resource allocation,'' \emph{European Journal of Operational Research}, vol. 270, no.~3, pp. 907--916, 2018.

\bibitem{SLW15a}
W.~Shi, Q.~Ling, G.~Wu, and W.~Yin, ``{EXTRA}: An exact first-order algorithm for decentralized consensus optimization,'' \emph{SIAM Journal on Optimization}, vol.~25, no.~2, pp. 944--966, 2015.

\bibitem{SiJ16}
A.~Simonetto and H.~Jamali-Rad, ``Primal recovery from consensus-based dual decomposition for distributed convex optimization,'' \emph{Journal of Optimization Theory and Applications}, vol. 168, pp. 172--197, 2016.

\bibitem{FMG17}
A.~Falsone, K.~Margellos, S.~Garatti, and M.~Prandini, ``Dual decomposition for multi-agent distributed optimization with coupling constraints,'' \emph{Automatica}, vol.~84, pp. 149--158, 2017.

\bibitem{LLS20}
C.~Liu, H.~Li, and Y.~Shi, ``A unitary distributed subgradient method for multi-agent optimization with different coupling sources,'' \emph{Automatica}, vol. 114, no. 108834, pp. 1--13, 2020.

\bibitem{LWY21}
S.~Liang, L.~Y. Wang, and G.~Yin, ``Distributed dual subgradient algorithms with iterate-averaging feedback for convex optimization with coupled constraints,'' \emph{IEEE Transactions on Cybernetics}, vol.~51, no.~5, pp. 2529--2539, 2021.

\bibitem{LBN24}
H.~Liu, S.~Bose, H.~D. Nguyen, Y.~Guo, T.~T. Doan, and C.~L. Beck, ``Distributed dual subgradient methods with averaging and applications to grid optimization,'' \emph{Journal of Optimization Theory and Applications}, vol. 203, no.~2, pp. 1991--2024, 2024.

\bibitem{MaC17}
D.~Mateos-Núñez and J.~Cortés, ``Distributed saddle-point subgradient algorithms with {Laplacian} averaging,'' \emph{IEEE Transactions on Automatic Control}, vol.~62, no.~6, pp. 2720--2735, 2017.

\bibitem{CNS14}
T.-H. Chang, A.~Nedić, and A.~Scaglione, ``Distributed constrained optimization by consensus-based primal-dual perturbation method,'' \emph{IEEE Transactions on Automatic Control}, vol.~59, no.~6, pp. 1524--1538, 2014.

\bibitem{HMS24}
Y.~Huang, Z.~Meng, J.~Sun, and W.~Ren, ``A unified distributed method for constrained networked optimization via saddle-point dynamics,'' \emph{IEEE Transactions on Automatic Control}, vol.~69, no.~3, pp. 1818--1825, 2024.

\bibitem{YXJ24}
Y.~Huang, X.~Zeng, J.~Sun, and Z.~Meng, ``Distributed event-triggered algorithm for convex optimization with coupled constraints,'' \emph{Automatica}, vol. 170, no. 111877, pp. 1--10, 2024.

\bibitem{AyH19}
N.~S. Aybat and E.~Y. Hamedani, ``A distributed {ADMM}-like method for resource sharing over time-varying networks,'' \emph{SIAM Journal on Optimization}, vol.~29, no.~4, pp. 3036--3068, 2019.

\bibitem{LWY20}
S.~Liang, L.~Y. Wang, and G.~Yin, ``Distributed smooth convex optimization with coupled constraints,'' \emph{IEEE Transactions on Automatic Control}, vol.~65, no.~1, pp. 347--353, 2020.

\bibitem{HeZ24}
R.~Heusdens and G.~Zhang, ``Distributed optimisation with linear equality and inequality constraints using {PDMM},'' \emph{IEEE Transactions on Signal and Information Processing over Networks}, vol.~10, pp. 294--306, 2024.

\bibitem{GoZ23}
K.~Gong and L.~Zhang, ``Decentralized proximal method of multipliers for convex optimization with coupled constraints,'' 2023, \href{http://arxiv.org/abs/1409.0876}{arXiv:1409.0876}.

\bibitem{FaP23}
A.~Falsone and M.~Prandini, ``{Augmented Lagrangian Tracking} for distributed optimization with equality and inequality coupling constraints,'' \emph{Automatica}, vol. 157, no. 111269, pp. 1--13, 2023.

\bibitem{FaP22}
------, ``Distributed decision-coupled constrained optimization via {Proximal-Tracking},'' \emph{Automatica}, vol. 135, no. 109938, pp. 1--12, 2022.

\bibitem{WWL23}
X.~Wu, H.~Wang, and J.~Lu, ``Distributed optimization with coupling constraints,'' \emph{IEEE Transactions on Automatic Control}, vol.~68, no.~3, pp. 1847--1854, 2023.

\bibitem{SLW15b}
W.~Shi, Q.~Ling, G.~Wu, and W.~Yin, ``A proximal gradient algorithm for decentralized composite optimization,'' \emph{IEEE Transactions on Signal Processing}, vol.~63, no.~22, pp. 6013--6023, 2015.

\bibitem{FNN20}
A.~Falsone, I.~Notarnicola, G.~Notarstefano, and M.~Prandini, ``{Tracking-ADMM} for distributed constraint-coupled optimization,'' \emph{Automatica}, vol. 117, no. 108962, pp. 1--13, 2020.

\bibitem{AWL18}
N.~S. Aybat, Z.~Wang, T.~Lin, and S.~Ma, ``Distributed linearized alternating direction method of multipliers for composite convex consensus optimization,'' \emph{IEEE Transactions on Automatic Control}, vol.~63, no.~1, pp. 5--20, 2018.

\bibitem{MaO17}
A.~Makhdoumi and A.~Ozdaglar, ``Convergence rate of distributed {ADMM} over networks,'' \emph{IEEE Transactions on Automatic Control}, vol.~62, no.~10, pp. 5082--5095, 2017.

\bibitem{HoC17}
M.~Hong and T.-H. Chang, ``Stochastic proximal gradient consensus over random networks,'' \emph{IEEE Transactions on Signal Processing}, vol.~65, no.~11, pp. 2933--2948, 2017.

\bibitem{Ber09}
D.~Bertsekas, \emph{Convex Optimization Theory}, 1st~ed.\hskip 1em plus 0.5em minus 0.4em\relax Nashua, NH, USA: Athena Scientific, 2009.

\bibitem{SLY14}
W.~Shi, Q.~Ling, K.~Yuan, G.~Wu, and W.~Yin, ``On the linear convergence of the {ADMM} in decentralized consensus optimization,'' \emph{IEEE Transactions on Signal Processing}, vol.~62, no.~7, pp. 1750--1761, 2014.

\bibitem{WuL23}
X.~Wu and J.~Lu, ``A unifying approximate method of multipliers for distributed composite optimization,'' \emph{IEEE Transactions on Automatic Control}, vol.~68, no.~4, pp. 2154--2169, 2023.

\bibitem{Roc76}
R.~T. Rockafellar, ``Augmented {Lagrangians} and applications of the proximal point algorithm in convex programming,'' \emph{Mathematics of Operations Research}, vol.~1, no.~2, pp. 97--196, 1976.

\bibitem{Ber16}
D.~Bertsekas, \emph{Nonlinear Programming}, 3rd~ed.\hskip 1em plus 0.5em minus 0.4em\relax Nashua, NH, USA: Athena Scientific, 2016.

\bibitem{GoR01}
C.~Godsil and G.~Royle, \emph{Algebraic Graph Theory}.\hskip 1em plus 0.5em minus 0.4em\relax New York, NY, USA: Springer, 2001.

\bibitem{NOA17}
A.~Nedi\'{c}, A.~Olshevsky, and W.~Shi, ``Achieving geometric convergence for distributed optimization over time-varying graphs,'' \emph{SIAM Journal on Optimization}, vol.~27, no.~4, pp. 2597--2633, 2017.

\bibitem{YuN17}
H.~Yu and M.~J. Neely, ``A simple parallel algorithm with an {$O(1/t)$} convergence rate for general convex programs,'' \emph{SIAM Journal on Optimization}, vol.~27, no.~2, pp. 759--783, 2017.

\bibitem{CVXPY}
S.~Diamond and S.~Boyd, ``{CVXPY}: {A} {P}ython-embedded modeling language for convex optimization,'' \emph{Journal of Machine Learning Research}, vol.~17, no.~83, pp. 1--5, 2016.

\bibitem{Nes18a}
Y.~Nesterov, \emph{Lectures on Convex Optimization}, 2nd~ed.\hskip 1em plus 0.5em minus 0.4em\relax Cham, Switzerland: Springer, 2018.

\end{thebibliography}



\begin{IEEEbiographynophoto}{Zixuan Liu} received the B.S. degree in computer science and technology from ShanghaiTech University, Shanghai, China, in 2022. He is currently pursing the M.S. degree in computer science and technology at ShanghaiTech University, Shanghai, China. His research interests include distributed optimization and multi-agent decision making.
\end{IEEEbiographynophoto}

\begin{IEEEbiographynophoto}{Xuyang Wu} (Member, IEEE) received the B.S. degree in information and computing science from Northwestern Polytechnical University, Xi’an, China, in 2015, and the Ph.D. degree in communication and information systems from the University of Chinese Academy of Sciences, China, in 2020. He was a postdoctoral researcher at the Division of Decision and Control Systems, KTH, from 2020 to 2023. 

He is currently an assistant professor at the School of System Design and Intelligent Manufacturing, Southern University of Science and Technology. His research interests include distributed optimization and machine learning.
\end{IEEEbiographynophoto}

\begin{IEEEbiographynophoto}{Dandan Wang} received the B.S. degree in information and communication engineering from Donghua University, Shanghai, China, in 2018, and the Ph.D. degree in communication and information systems from the University of Chinese Academy of Sciences, China, in 2025. Her research interests include distributed optimization, online optimization, and their applications in wireless networks.
\end{IEEEbiographynophoto}

\newpage

\begin{IEEEbiographynophoto}{Jie Lu} (Member, IEEE) received the B.S. degree in information engineering from Shanghai Jiao Tong University, Shanghai, China, in 2007, and the Ph.D. degree in electrical and computer engineering from the University of Oklahoma, Norman, OK, USA, in 2011. 

She is currently an Associate Professor with the School of Information Science and Technology, ShanghaiTech University, Shanghai, China. Before she joined ShanghaiTech University in 2015, she was a Postdoctoral Researcher with the KTH Royal Institute of Technology, Stockholm, Sweden, and with the Chalmers University of Technology, Gothenburg, Sweden from 2012 to 2015. Her research interests include distributed optimization, optimization theory and algorithms, learning-assisted optimization, and networked dynamical systems.
\end{IEEEbiographynophoto}

\end{document}